\documentclass[a4paper,10pt]{amsart}
\usepackage{enumerate}
\usepackage{amsthm,amsmath,amssymb,graphics,graphicx,hyperref,epstopdf,mathrsfs}
\usepackage{breqn}           
\title{Structural Decomposition of Monomial Resolutions}
\author{Guillermo Alesandroni}
\address{Department of Mathematics, Wake Forest University, 1834 Wake Forest Rd, Winston-Salem, NC 27109}
\email{alesangc@wfu.edu}

\newtheorem{theorem}{Theorem}[section]
\newtheorem{proposition}[theorem]{Proposition}
\newtheorem{corollary}[theorem]{Corollary}
\newtheorem{lemma}[theorem]{Lemma}

\newtheorem{conjecture}[theorem]{Conjecture}

\theoremstyle{definition}
\newtheorem{definition}[theorem]{Definition}

\newtheorem{example}[theorem]{Example}
\newtheorem{construction}[theorem]{Construction}

\DeclareMathOperator{\betti}{b}
\DeclareMathOperator{\pd}{pd}
\DeclareMathOperator{\mdeg}{mdeg}
\DeclareMathOperator{\lcm}{lcm}

\DeclareMathOperator{\hdeg}{hdeg}

\DeclareMathOperator{\rank}{rank}

\begin{document}
\maketitle
\begin{abstract}
We express the multigraded Betti numbers of an arbitrary monomial ideal in terms of the multigraded Betti numbers of two basic classes of ideals. This decomposition has multiple applications. In some concrete cases, we use it to construct minimal resolutions of classes of monomial ideals; in other cases, we use it to compute projective dimensions. To illustrate the effectiveness of the structural decomposition,  we give a new proof of a classic theorem by Charalambous that states the following: let $k$ be a field, and $M$ an Artinian monomial ideal in $S=k[x_1,\ldots,x_n]$; then, for all $i$, $\betti_i(S/M) \geq {n \choose i }$. 
 \end{abstract}
 
 \section{Introduction}
The problem of finding the minimal resolution of an arbitrary monomial ideal in closed form has been deemed utopic by many a mathematician. As a consequence, people have tried to restrict the study of minimal resolutions to particular classes of ideals. Borel ideals, minimally resolved by the Eliahou-Kervaire resolution [EK]; generic ideals, minimally resolved by the Scarf complex [BPS]; and dominant ideals, minimally resolved by the Taylor resolution [Al], are examples of this restrictive approach.
 
 In the first half of this paper, however, we turn to the general problem, and decompose the minimal resolution of an arbitrary monomial ideal in terms of the minimal resolutions of two basic classes that we call dominant, and purely nondominant ideals. More precisely, we express the multigraded Betti numbers of an ideal as the sum of the multigraded Betti numbers of some dominant and some purely nondominant ideals. Since dominant ideals are minimally resolved by their Taylor resolutions, our decomposition reduces the study of minimal monomial resolutions to the study of minimal resolutions of purely nondominant ideals.
 
 Unfortunately, the resolutions of purely nondomiant ideals involve the same challenges that we encounter in the general context. Some of these difficulties are the existence of ghost terms, characteristic dependence, and the striking fact that some of the simplest purely nondominant ideals cannot be minimally resolved by any subcomplex of the Taylor resolution. Thus, in the second half of this work we focus our efforts on one particular case: monomial ideals whose structural decomposition has no purely nondominant part. As a result of this study, we obtain the multigraded Betti numbers of two families that we call $2$-semidominant and almost generic ideals.
 
 The structural decomposition is also a useful tool to compute projective dimensions. We prove, for instance, that if an ideal $M$ satisfies certain conditions, $\pd(S/M)=2$, and, under some other conditions, $\pd(S/M)=n$, where $n$ is the number of variables in the polynomial ring. Another result, also related to projective dimensions, is a new proof of a classic theorem of Charalambous [Ch] (see also [Pe, Corollary 21.6]), stating: let $k$ be a field, and $M$ an Artinian monomial ideal in $S=k[x_1,\ldots,x_n]$; then, for all $i$, $\betti_i(S/M) \geq {n \choose i }$. While the original proof relies on the radical of an ideal, ours is based on the structural decomposition.
  
 The organization of the article is as follows. Section 2 is about background and notation. Sections 3 and 4 are technical. They contain some isomorphism theorems, as well as the structural decomposition theorems advertised above. In section 5, we compute the multigraded Betti numbers of two families of ideals. In section 6, we compute projective dimensions. Section 7 is the conclusion; it includes some comments, questions, and conjectures.
 
 \section{Background and Notation}
 Throughout this paper $S$ represents a polynomial ring over an arbitrary field $k$, in a finite number variables. The letter $M$ always denotes a monomial ideal
in $S$. With minor modifications, the constructions that we give below can be found in [Me,Pe].
 
\begin{construction}
Let $M$ be generated by a set of monomials $\{l_1,\ldots,l_q\}$. For every subset $\{l_{i_1},\ldots,l_{i_s}\}$ of $\{l_1,\ldots,l_q\}$, with $1\leq i_1<\ldots<i_s\leq q$, 
we create a formal symbol $[l_{i_1},\ldots,l_{i_s}]$, called a \textbf{Taylor symbol}. The Taylor symbol associated to $\{\}$ is denoted by $[\varnothing]$.
For each $s=0,\ldots,q$, set $F_s$ equal to the free $S$-module with basis $\{[l_{i_1},\ldots,l_{i_s}]:1\leq i_1<\ldots<i_s\leq q\}$ given by the 
${q\choose s}$ Taylor symbols corresponding to subsets of size $s$. That is, $F_s=\bigoplus\limits_{i_1<\ldots<i_s}S[l_{i_1},\ldots,l_{i_s}]$ 
(note that $F_0=S[\varnothing]$). Define
\[f_0:F_0\rightarrow S/M\]
\[s[\varnothing]\mapsto f_0(s[\varnothing])=s\]
For $s=1,\ldots,q$, let $f_s:F_s\rightarrow F_{s-1}$ be given by
\[f_s\left([l_{i_1},\ldots,l_{i_s}]\right)=
 \sum\limits_{j=1}^s\dfrac{(-1)^{j+1}\lcm(l_{i_1},\ldots,l_{i_s})}{\lcm(l_{i_1},\ldots,\widehat{l_{i_j}},\ldots,l_{i_s})}
 [l_{i_1},\ldots,\widehat{l_{i_j}},\ldots,l_{i_s}]\]
 and extended by linearity.
 The \textbf{Taylor resolution} $\mathbb{T}_{l_1,\ldots,l_q}$ of $S/M$ is the exact sequence
 \[\mathbb{T}_{l_1,\ldots,l_q}:0\rightarrow F_q\xrightarrow{f_q}F_{q-1}\rightarrow\cdots\rightarrow F_1\xrightarrow{f_1}F_0\xrightarrow{f_0} 
 S/M\rightarrow0.\]
 \end{construction}
We define the \textbf{multidegree} of a Taylor symbol $[l_{i_1},\ldots,l_{i_s}]$, denoted $\mdeg[l_{i_1},\ldots,l_{i_s}]$, as follows:
  $\mdeg[l_{i_1},\ldots,l_{i_s}]=\lcm(l_{i_1},\ldots,l_{i_s})$. The Taylor symbols $[l_{i_1},\ldots,l_{i_s}]$ are called \textbf{faces}. A Taylor symbol
  of the form $[l_{i_1},\ldots,\widehat{l_{i_j}},\ldots, l_{i_s}]$ is referred to as a \textbf{facet} of the face $[l_{i_1},\ldots,l_{i_s}]$.
  
  \textit{Note}:
  In our construction above, the generating set $\{l_1,\ldots,l_q\}$ is not required to be minimal. Thus, $S/M$ has many Taylor resolutions. We reserve the notation 
  $\mathbb{T}_M$ for the Taylor resolution of $S/M$, determined by the minimal generating set of $M$. (Although some authors define a single Taylor resolution of $S/M$, our construction is general, like in [Ei].)

\begin{construction}
Let $M$ be minimally generated by $\{l_1,\ldots,l_q\}$. Let $A$ be the set of Taylor symbols of $\mathbb{T}_M$ whose 
multidegrees are not common to other Taylor symbols; that is, a Taylor symbol $[\sigma]$ is in $A$ if and only if $\mdeg[\sigma]\neq \mdeg[\sigma']$,
for every Taylor symbol $[\sigma']\neq [\sigma]$. For each $s=0,\ldots,q$, set $G_s$ equal to the free $S$-module with basis 
$\{[l_{i_1},\ldots,l_{i_s}]\in A:1\leq i_1<\ldots<i_s\leq q\}$. For each $s=0,\ldots,q$, let $g_s=f_s\restriction_{G_s}$. It can be proven that the $g_s$
are well defined (more precisely, that $g_s\left(G_s\right)\subseteq G_{s-1}$) and that
\[0\rightarrow G_q\xrightarrow{g_q}G_{q-1}\rightarrow \cdots\rightarrow G_1\xrightarrow{g_1}G_0\xrightarrow{g_0} S/M\rightarrow 0\]
is a subcomplex of $\mathbb{T}_M$, which is called the \textbf{Scarf complex} of $S/M$. 
\end{construction}

\begin{definition}

 Let $M$ be a monomial ideal, and let
 \[\mathbb{F}:\cdots\rightarrow F_i\xrightarrow{f_i}F_{i-1}\rightarrow\cdots\rightarrow F_1\xrightarrow{f_1}F_0\xrightarrow{f_0} S/M\rightarrow 0\]
be a free resolution of $S/M$. 
We say that a basis element $[\sigma]$ of $\mathbb{F}$ has \textbf{homological degree i}, denoted $\hdeg[\sigma]=i$, if 
$[\sigma] \in F_i$. $\mathbb{F}$ is said to be a \textbf{minimal resolution} if for every $i$, the differential matrix $\left(f_i\right)$ of $\mathbb{F}$
has no invertible entries.
\end{definition}
\begin{definition}
Let $M$ be a monomial ideal, and let
 \[\mathbb{F}:\cdots\rightarrow F_i\xrightarrow{f_i}F_{i-1}\rightarrow\cdots\rightarrow F_1\xrightarrow{f_1}F_0\xrightarrow{f_0} S/M\rightarrow 0\]
be a minimal free resolution of $S/M$.
\begin{itemize}
 \item For every $i\geq 0$, the $i^{th}$ \textbf{Betti number} $\betti_i\left(S/M\right)$ of $S/M$ is $\betti_i\left(S/M\right)=\rank(F_i)$.
\item For every $i\geq 0$, and every monomial $l$, the \textbf{multigraded Betti number} $\betti_{i,l}\left(S/M\right)$ of $S/M$, in homological degree $i$ and multidegree $l$,
is \[\betti_{i,l}\left(S/M\right)=\#\{\text{basis elements }[\sigma]\text{ of }F_i:\mdeg[\sigma]=l\}.\]
\item The \textbf{projective dimension} $\pd\left(S/M\right)$ of $S/M$ is \[\pd\left(S/M\right)=\max\{i:\betti_i\left(S/M\right)\neq 0\}.\]
\end{itemize}
\end{definition}

\begin{definition}
Let $M$ be minimally generated by a set of monomials $G$.
\begin{itemize}
\item A monomial $m\in G$ is called \textbf{dominant} (in $G$) if there is a variable $x$, such that for all $m'\in G\setminus\{m\}$, the exponent with 
which $x$ appears in the factorization of $m$ is larger than the exponent with which $x$ appears in the factorization of $m'$.
The set $G$ is called \textbf{dominant} if each of its elements is dominant. The ideal $M$ is called \textbf{dominant} if $G$ is dominant.
\item $G$ is called \textbf{$p$-semidominant} if $G$ contains
exactly $p$ nondominant monomials. The ideal $M$ is \textbf{$p$-semidominant} if $G$ is $p$-semidominant.
\item  We say that $G$ is \textbf{purely nondominant} when all the elements of $G$ are nondominant. In this case, we also say that $M$ is \textbf{purely nondominant}.
\end{itemize}
\end{definition}

 \begin{example}\label{example 1}
  Let $M_1$, $M_2$, and $M_3$ be minimally generated by $G_1=\{a^2,b^3,ab\}$, $G_2=\{ab,bc,ac\}$, and $G_3=\{a^2b,ab^3c,bc^2\}$, respectively. Note that $a^2$ and $b^3$ are dominant in $G_1$, but $ab$ is not. Thus, both the set $G_1$ and the ideal $M_1$ are {1}-semidominant. On the other hand, $ab$, $bc$, and $ac$ are nondominant in $G_2$. Therefore, $G_2$ and $M_2$ are purely nondominant (as well as {3}-semidominant). Finally, $a^2b$, $ab^3c$, and $bc^2$ are dominant in $G_3$. Thus, $G_3$ and $M_3$ are dominant.
\end{example}

 \section{Isomorphism Theorems}

The notation that we introduce below retains its meaning until the end of this section.
Let $M$ be a monomial ideal with minimal generating set $G=\{m_1,\ldots,m_q,n_1,\ldots,n_p\}$, where $m_1,\ldots,m_q$ are dominant, and $n_1,\ldots,n_p$ are nondominant. Let $1\leq d \leq q$, and let 
$H=\{h_1,\ldots,h_c\}=\{m_{d+1},\ldots,m_q,n_1,\ldots,n_p\}$. Then $G$ can be expressed in the form $G=\{m_1,\ldots,m_d,h_1,\ldots,h_c\}$.
Let $m=\lcm(m_{r_1},\ldots,m_{r_j})$, where $1\leq r_1 <\ldots< r_j \leq d$. By convention, if $j=0$, $m=1$. For all $s=1,\ldots,c$, let $h'_s=\dfrac{\lcm(m,h_s)}{m}$. Let $M_m=(h'_1,\ldots,h'_c)$.

\begin{example}\label{example 1}
  Let $M=(a^3b^2,c^3d,ac^2,a^2c,b^2d,abc,bcd)$. Note that $M$ is $5$-semidominant, with $m_1=a^3b^2$, $m_2=c^3d$, $n_1=ac^2$, $n_2=a^2c$, $n_3=b^2d$, $n_4=abc$, $n_5=bcd$. If we set $d=2$, then $H=\{h_1,\ldots,h_5\}=\{n_1,\ldots,n_5\}$. Suppose that $m=\lcm(m_1)=a^3b^2$. Then $M_m=(h'_1,\ldots,h'_5)$, where $h'_1=\dfrac{\lcm(a^3b^2,ac^2)}{a^3b^2}=c^2$; $h'_2=\dfrac{\lcm(a^3b^2,a^2c)}{a^3b^2}=c$; 
$h'_3=\dfrac{\lcm(a^3b^2,b^2d)}{a^3b^2}=d$; $h'_4=\dfrac{\lcm(a^3b^2,abc)}{a^3b^2}=c$; $h'_5=\dfrac{\lcm(a^3b^2,bcd)}{a^3b^2}=cd$.  
Thus, $M_{a^3b^2}=(c^2,c,d,c,cd)$. Although $\{c^2,c,d,c,cd\}$ does not generate $M_{a^3b^2}$ minimally, sometimes, nonminimal generating sets like this will serve our purpose.
\end{example}

\begin{proposition}\label{1}
Let $1\leq s_1<\ldots<s_i \leq c$. The Taylor symbols $[h'_{s_1},\ldots,h'_{s_i}]$ of $\mathbb{T}_{h'_1,\ldots,h'_c}$, and $[m_{r_1},\ldots,m_{r_j},h_{s_1},\ldots,h_{s_i}]$ of $\mathbb{T}_M$, satisfy
\[\mdeg[m_{r_1},\ldots,m_{r_j},h_{s_1},\ldots,h_{s_i}]=m \mdeg[h'_{s_1},\ldots,h'_{s_i}].\]
\end{proposition}

\begin{proof}
Note that $h_{s_1}\mid \lcm(m,h_{s_1})=mh'_{s_1} \text{, and } mh'_{s_1}\mid m\lcm(h'_{s_1},\ldots,h'_{s_i})$. Thus, $h_{s_1}\mid m \lcm(h'_{s_1},\ldots,h'_{s_i})$.\\
Similarly, 
$h_{s_2},\ldots,h_{s_i}\mid m \lcm(h'_{s_1},\ldots,h'_{s_i})$. Hence, $\lcm(m,h_{s_1},\ldots,h_{s_i})\mid m \lcm(h'_{s_1},\ldots,h'_{s_i})$. We will show that $m\lcm(h'_{s_1},\ldots,h'_{s_i})\mid \lcm(m,h_{s_1},\ldots,h_{s_i})$.\\
 Let $h'_{s_1}=x_1^{\alpha_{11}}\ldots x_n^{\alpha_{1n}},\ldots , h'_{s_i}=x_1^{\alpha_{i1}}\ldots x_n^{\alpha_{in}}$,
 and let $\gamma_1=\max(\alpha_{11},\ldots,\alpha_{i1}), \ldots, \gamma_n = \max(\alpha_{1n}, \ldots, \alpha_{in})$. Then $\lcm(h'_{s_1}, \ldots, h'_{s_i})=x_1^{\gamma_1} \ldots x_n^{\gamma_n}$. Notice that
  $m x_1^{\gamma_1}$ divides one of $m h'_{s_1} = \lcm(m,h_{s_1}),  \ldots,   m h'_{s_i} = \lcm(m,h_{s_i})$, and therefore,
  $m x_1^{\gamma_1} \mid  \lcm(m,h_{s_1},\ldots,h_{s_i})$.
Similarly,
$m x_2^{\gamma_2},\ldots,mx_n^{\gamma_n} \mid  \lcm(m,h_{s_1},\ldots,h_{s_i})$.
Thus,
$x_1^{\gamma_1},\ldots, x_n^{\gamma_n} \mid \dfrac{\lcm(m,h_{s_1}, \ldots, h_{s_i})}{m}$.
It follows that
$\lcm(h'_{s_1}, \ldots, h'_{s_i})=x_1^{\gamma_1} \ldots x_n^{\gamma_n} \mid \dfrac{\lcm(m,h_{s_1}, \ldots, h_{s_i})}{m}$, which is equivalent to saying that
$m\lcm(h'_{s_1},\ldots,h'_{s_i}) \mid \lcm(m,h_{s_1},\ldots,h_{s_i})$.
Finally,
\begin{dmath*}
m\mdeg[h'_{s_1}, \ldots, h'_{s_i}]=m \lcm(h'_{s_1}, \ldots, h'_{s_i})=\lcm(m,h_{s_1},\ldots,h_{s_i})=
\lcm(\lcm(m_{r_1}, \ldots,m_{r_j}),h_{s_1}, \ldots, h_{s_i}) = \lcm(m_{r_1}, \ldots,m_{r_j},h_{s_1}, \ldots, h_{s_i}) = \mdeg[m_{r_1}, \ldots,m_{r_j},h_{s_1}, \ldots, h_{s_i}].
\end{dmath*}
  \end{proof}
  
  \begin{example}\label{example 2}
 Let $M$, $m$, and $H$ be as in Example \ref{example 1}. The Taylor symbols $[h'_2,h'_3]=[c,d]$ of $\mathbb{T}_{c^2,c,d,c,cd}$, and $[m_1,h_2,h_3]=[a^3b^2,a^2c,b^2d]$ of $\mathbb{T}_M$, have multidegrees $cd$ and $a^3b^2cd$, respectively. Therefore, $\mdeg[a^3b^2,c,d]=a^3b^2 \mdeg[c,d]$, which is consistent with Proposition \ref{1}.    
\end{example}

\textit{Note}:
 We will say that a monomial $l$ \textbf{occurs} in a resolution $\mathbb{F}$ if there is a basis element of $\mathbb{F}$ with multidegree $l$.
If $a$ is an entry of a differential matrix of a resolution $\mathbb{F}$ and $[\sigma]$ is an element of the basis of $\mathbb{F}$, by abusing the 
language we will often say that \textbf{$a$ is an entry of $\mathbb{F}$} and \textbf{$[\sigma]$ is an element of $\mathbb{F}$}. Moreover, sometimes we will use the notation
$[\sigma]\in \mathbb{F}$.

\begin{theorem}\label{3}
 Let $m'$ be a multidegree that occurs in $\mathbb{T}_{h'_1,\ldots,h'_c}$.
 \begin{enumerate}[(i)]
  \item There are no basis elements of $\mathbb{T}_M$, with multidegree $mm'$ and homological degree less than $j$.
  \item For every $i$, there is a bijective correspondence between the basis elements of $\mathbb{T}_{h'_1,\ldots,h'_c}$, with multidegree $m'$ and homological 
  degree $i$, and the basis elements of $\mathbb{T}_M$, with multidegree $mm'$ and homological degree $i+j$.
 \end{enumerate}
 \end{theorem}
 
\begin{proof}
 (i) Let $[h'_{s_1},\ldots,h'_{s_i}]$ be a basis element of $\mathbb{T}_{h'_1,\ldots,h'_c}$, with multidegree $m'$. By Proposition \ref{1},
 \[mm'=m\mdeg[h'_{s_1},\ldots,h'_{s_i}]=\mdeg[m_{r_1},\ldots,m_{r_j},h_{s_1},\ldots,h_{s_i}].\]
 It follows that every basis element $[\sigma]$ of $\mathbb{T}_M$, with multidegree $mm'$, must contain the same dominant monomials $m_{r_1},\ldots,m_{r_j}$ [Al, Lemma 4.3]. This means that
 $\hdeg[\sigma]\geq j$.\\
 (ii) Let $A_{i,m'}=\left\{[\sigma]\in \mathbb{T}_{h'_1,\ldots,h'_c}: \hdeg[\sigma]=i; \mdeg[\sigma]=m'\right\}$.\\
 Let $B_{i,m'}=\left\{[\sigma]\in \mathbb{T}_M: \hdeg[\sigma]=i+j; \mdeg[\sigma]=mm'\right\}$.\\ 
 Let $f_{i,m'}:A_{i,m'}\rightarrow B_{i,m'}$ be defined by $f_{i,m'}[h'_{s_1},\ldots,h'_{s_i}]=[m_{r_1},\ldots,m_{r_j},h_{s_1},\ldots,h_{s_i}]$.\\
 Notice that $f_{i,m'}$ is well defined:\\ if $[\sigma]\in A_{i,m'}$, then $\hdeg f_{i,m'}[\sigma]=i+j$ and by Proposition \ref{1}, $mm'=m\mdeg[\sigma]=\mdeg f_{i,m'}[\sigma]$.\\
 Besides that, $f_{i,m'}$ is one to one:\\
 $f_{i,m'}[h'_{s_1},\ldots,h'_{s_i}]=f_{i,m'}[h'_{t_1},\ldots,h'_{t_i}]\\
 \Rightarrow  [m_{r_1},\ldots,m_{r_j},h_{s_1},\ldots,h_{s_i}]=[m_{r_1},\ldots,m_{r_j},h_{t_1},\ldots,h_{t_i}]\\
 \Rightarrow h_{s_1}=h_{t_1},\cdots,h_{s_i}=h_{t_i}\\
 \Rightarrow [ h'_{s_1},\ldots,h'_{s_i}]=[h'_{t_1},\ldots,h'_{t_i}].$\\
 Finally, $f_{i,m'}$ is onto:\\
 Suppose that $[\tau]$ is in $B_{i,m'}$. Let $[h'_{t_1},\ldots,h'_{t_k}]$ be an element in $\mathbb{T}_{h'_1,\ldots,h'_c}$, with multidegree $m'$. By Proposition \ref{1}, 
 \[mm'=m\mdeg[h'_{t_1},\ldots,h'_{t_k}]=\mdeg[m_{r_1},\ldots,m_{r_j},h_{t_1},\ldots,h_{t_k}].\] 
 Since $[\tau]$ and $[m_{r_1},\ldots,m_{r_j},h_{t_1},\ldots,h_{t_k}]$ are basis elements of equal multidegree, they must contain the same dominant monomials [Al, Lemma 4.3] and, given that $[\tau]$ has homological degree $i+j$, $[\tau]$ must be of the form $[\tau]=[m_{r_1},\ldots,m_{r_j},h_{s_1},\ldots,h_{s_i}]$. By Proposition \ref{1}, $[h'_{s_1},\ldots,h'_{s_i}] \in A_{i,m'}$, and $f_{i,m'} [h'_{s_1},\ldots,h'_{s_i}]=[\tau]$.
\end{proof}

\begin{example}\label{example 3}
Let $M$, $m$, and $H$ be as in Example \ref{example 1}. Let $m'=cd$. Since the only basis element of $\mathbb{T}_M$ in homological degree $0$ is 
$[\varnothing]$, and $\mdeg[\varnothing]=1$, there are no basis elements of $\mathbb{T}_M$ in homological degree $0$ and multidegree $mm'=a^3b^2cd$. This illustrates Theorem \ref{3}(i). On the other hand, the basis elements of 
$\mathbb{T}_{c^2,c,d,c,cd}$ with multidegree $m'=cd$ are \\
$[h'_5]$, in homological degree $1$; \\
$[h'_2,h'_5]$, $[h'_3,h'_5]$, $[h'_4,h'_5]$, $[h'_2,h'_3]$, $[h'_3,h'_4]$ in homological degree $2$;\\
 $[h'_2,h'_3,h'_5]$, $[h'_2,h'_4,h'_5]$, $[h'_3,h'_4,h'_5]$, $[h'_2,h'_3,h'_4]$ in homological degree $3$; and \\
  $[h'_2,h'_3,h'_4,h'_5]$ in homological degree $4$.\\
   Similarly, the basis elements of $\mathbb{T}_M$ with multidegree $mm'=a^3b^2cd$ are \\
   $[m_1,h_5]$ in homological degree $2$;\\
    $[m_1,h_2,h_5]$, $[m_1,h_3,h_5]$, $[m_1,h_4,h_5]$, $[m_1,h_2,h_3]$, $[m_1,h_3,h_4]$, in homological degree $3$; \\
    $[m_1,h_2,h_3,h_5]$, $[m_1,h_2,h_4,h_5]$, $[m_1,h_3,h_4,h_5]$, $[m_1,h_2,h_3,h_4]$ in homological degree $4$; and \\
    $[m_1,h_2,h_3,h_4,h_5]$ in homological degree $5$, which illustrates the bijective correspondence of Theorem \ref{3}(ii).  
\end{example}

Notation: for every multidegree $m'$ that occurs in $\mathbb{T}_{h'_1,\ldots,h'_c}$ and every $i=0,\ldots,c$, let $f_{i,m'}:A_{i,m'}\rightarrow B_{i,m'}$ be
the bijection constructed in Theorem \ref{3}. 
Let $A_i=\bigcup\limits_{m'}A_{i,m'}$ and $B_i=\bigcup\limits_{m'}B_{i,m'}$. Let us define $f_i:A_i\rightarrow B_i$ by
$f_i[\sigma]=f_{i,m'}[\sigma]$, if $\mdeg[\sigma]=m'$.
Let $A=\bigcup\limits_i A_i$; $B=\bigcup\limits_i B_i$. Let us define $f:A\rightarrow B$ by $f[\sigma]=f_i[\sigma]$, if $\hdeg[\sigma]=i$.
Note that $A$ is the basis of $\mathbb{T}_{h'_1,\ldots,h'_c}$, and $f$ is a bijection that sends an element with multidegree $m'$ and homological degree $i$
to an element with multidegree $mm'$ and homological degree $i+j$.

To better understand the statement of the next theorem, we refer the reader to [Al, Remark 3.4].

\begin{theorem}\label{4}
 If $a_{\pi\theta}$ is an entry of $\mathbb{T}_{h'_1,\ldots,h'_c}$, determined by elements $[\theta],[\pi]\in A$, then $f[\theta],f[\pi]$ determine an entry
 $b_{\pi\theta}$ of $\mathbb{T}_M$ such that $b_{\pi\theta}=(-1)^ja_{\pi\theta}$.
\end{theorem}

\begin{proof}
Since $[\theta],[\pi]$ appear in consecutive homological degrees, so do $f[\theta],f[\pi]$. Thus, $f[\theta],f[\pi]$ determine an entry $b_{\pi\theta}$ of
$\mathbb{T}_M$.

If $[\pi]$ is a facet of $[\theta]$, then $f[\pi]$ is also a facet of $f[\theta]$ and, these elements are of the form:
\[[\theta]=[h'_{s_1},\ldots,h'_{s_i}];\quad [\pi]=[h'_{s_1},\ldots,\widehat{h'_{s_t}},\ldots,h'_{s_i}];\]
\[f[\theta]=[m_{r_1},\ldots,m_{r_j},h_{s_1},\ldots,h_{s_i}];\quad f[\pi]=[m_{r_1},\ldots,m_{r_j},h_{s_1},\ldots,\widehat{h_{s_t}},\ldots,h_{s_i}]\]
Thus
$b_{\pi\theta}=(-1)^{j+t+1}\dfrac{\mdeg f[\theta]}{\mdeg f[\pi]}=(-1)^j(-1)^{t+1}\dfrac{m\mdeg[\theta]}{m\mdeg[\pi]}=(-1)^j a_{\pi\theta}$.

On the other hand, if $[\pi]$ is not a facet of $[\theta]$, $f[\pi]$ cannot be a facet of $f[\theta]$, either. Thus $a_{\pi\theta}=0=b_{\pi\theta}$.
\end{proof}

Notation: let $\mathbb{F}_0=\mathbb{T}_{h'_1,\ldots,h'_c}$. If there is an invertible entry $a^{(0)}_{\pi_0\theta_0}$ of $\mathbb{F}_0$, determined by elements
$[\theta_0],[\pi_0]\in \mathbb{F}_0$, let $\mathbb{F}_1$ be the resolution of $S/M_m$ such that
\[\mathbb{F}_0=\mathbb{F}_1\oplus \left(0\rightarrow S[\theta_0]\rightarrow S[\pi_0]\rightarrow 0\right).\]
Let us assume that $\mathbb{F}_{k-1}$ has been defined. If there is an invertible entry $a^{(k-1)}_{\pi_{k-1}\theta_{k-1}}$ of $\mathbb{F}_{k-1}$, 
determined by elements $[\theta_{k-1}],[\pi_{k-1}]$ of $\mathbb{F}_{k-1}$, let $\mathbb{F}_k$ be the resolution of $S/M_m$ such that
\[\mathbb{F}_{k-1}=\mathbb{F}_k \oplus \left(0\rightarrow S[\theta_{k-1}]\rightarrow S[\pi_{k-1}]\rightarrow 0\right).\]

\begin{theorem}\label{5}
 Suppose that $\mathbb{F}_0,\ldots,\mathbb{F}_u$ are resolutions of $S/M_m$, defined as above. Then
 \begin{enumerate}[(i)]
  \item It is possible to define resolutions $\mathbb{G}_0,\ldots,\mathbb{G}_u$ of $S/M$, as follows:
  \[\mathbb{G}_0=\mathbb{T}_M;\quad \mathbb{G}_{k-1}=\mathbb{G}_k\oplus \left(0\rightarrow Sf[\theta_{k-1}]\rightarrow Sf[\pi_{k-1}]\rightarrow 0\right).\]
  \item If $a^{(u)}_{\tau\sigma}$ is an entry of $\mathbb{F}_u$, determined by elements $[\sigma],[\tau]$ of $\mathbb{F}_u$, then 
  $f[\sigma],f[\tau]$ are in the basis of $\mathbb{G}_u$ and determine an entry $b^{(u)}_{\tau\sigma}$ of $\mathbb{G}_u$, such that 
  $b^{(u)}_{\tau\sigma}=(-1)^ja^{(u)}_{\tau\sigma}$.
 \end{enumerate}
\end{theorem}

\begin{proof}
 The proof is by induction on $u$. If $u=0$, (i) and (ii) are the content of Theorem \ref{4}.\\
 Let us assume that parts (i) and (ii) hold for $u-1$.
 We will prove parts (i) and (ii) for $u$.\\
 (i) We need to show that $\mathbb{G}_u$ can be defined by the rule
 \[\mathbb{G}_{u-1}=\mathbb{G}_u\oplus\left(0\rightarrow Sf[\theta_{u-1}]\rightarrow Sf[\pi_{u-1}]\rightarrow 0\right).\]
 In other words, we must show that $f[\theta_{u-1}]$, $f[\pi_{u-1}]$ are in the basis of $\mathbb{G}_{u-1}$, and the entry $b^{(u-1)}_{\pi_{u-1}\theta_{u-1}}$ of
 $\mathbb{G}_{u-1}$, determined by them, is invertible. But this follows from induction hypothesis and the fact that $a^{(u-1)}_{\pi_{u-1}\theta_{u-1}}$ is
 invertible.\\
 (ii) Notice that the basis of $\mathbb{F}_u$ is obtained from the basis of $\mathbb{F}_{u-1}$, by eliminating $[\theta_{u-1}],[\pi_{u-1}]$. This means
 that $[\sigma],[\tau]$ are in the basis of $\mathbb{F}_{u-1}$, and the pairs $\left([\sigma],[\tau]\right)$, $\left([\theta_{u-1}],[\pi_{u-1}]\right)$ are
 disjoint. Then by induction hypothesis, $f[\sigma]$, $f[\tau]$ are in the basis of $\mathbb{G}_{u-1}$, and because $f$ is a bijection, 
 $\left(f[\sigma],f[\tau]\right)$, $\left(f[\theta_{u-1}],f[\pi_{u-1}]\right)$ are disjoint pairs. Since the basis of $\mathbb{G}_u$ is obtained from the
 basis of $\mathbb{G}_{u-1}$, by eliminating $f[\theta_{u-1}],f[\pi_{u-1}]$, we must have that $f[\sigma],f[\tau]$ are in the basis of $\mathbb{G}_u$.
 Finally, we need to prove that $b^{(u)}_{\tau\sigma}=(-1)^ja^{(u)}_{\tau\sigma}$. By [Al, Lemma 3.2(iv)], if 
 $\hdeg[\sigma]\neq\hdeg[\theta_{u-1}]$, then $a^{(u)}_{\tau\sigma}=a^{(u-1)}_{\tau\sigma}$. In this case, we must also have that 
 $\hdeg f[\sigma]\neq \hdeg f[\theta_{u-1}]$, which implies that $b^{(u)}_{\tau\sigma}=b^{(u-1)}_{\tau\sigma}$, by the same lemma. Then, by induction hypothesis, 
 $b^{(u)}_{\tau\sigma}=b^{(u-1)}_{\tau\sigma}=(-1)^ja^{(u-1)}_{\tau\sigma}=(-1)^ja^{(u)}_{\tau\sigma}$. On the other hand, if 
 $\hdeg[\sigma]=\hdeg[\theta_{u-1}]$, then $\hdeg f[\sigma]=\hdeg f[\theta_{u-1}]$. Combining the induction hypothesis with [Al, Lemma 3.2(iii)], we 
 obtain
 \[b^{(u)}_{\tau\sigma}=b^{(u-1)}_{\tau\sigma}-\dfrac{b^{(u-1)}_{\tau\theta_{u-1}}b^{(u-1)}_{\pi_{u-1}\sigma}}{b^{(u-1)}_{\pi_{u-1}\theta_{u-1}}}=
 (-1)^j\left(a^{(u-1)}_{\tau\sigma}-\dfrac{a^{(u-1)}_{\tau\theta_{u-1}}a^{(u-1)}_{\pi_{u-1}\sigma}}{a^{(u-1)}_{\pi_{u-1}\theta_{u-1}}}\right)=(-1)^j
 a^{(u)}_{\tau\sigma}\]
\end{proof}

Since the process of making standard cancellations must eventually terminate, there is an integer $u\geq0$, such that $\mathbb{F}_0,\ldots,\mathbb{F}_u$ are
defined as above and $\mathbb{F}_u$ is a minimal resolution of $S/M_m$. For the rest of this section $u$ is such an integer and 
$\mathbb{F}_u$ is such a minimal resolution. Moreover, the resolutions $\mathbb{F}_0,\ldots,\mathbb{F}_u$ and $\mathbb{G}_0,\ldots,\mathbb{G}_u$ are also 
fixed for the rest of this section.

Notation: Let $A'=A\setminus\{[\theta_0],[\pi_0],\ldots,[\theta_{u-1}],[\pi_{u-1}]\}$ and 
$B'=B\setminus\{f[\theta_0],f[\pi_0],\ldots,f[\theta_{u-1}],$\\
$f[\pi_{u-1}]\}$. Notice that $A'$ is the basis of the minimal resolution $\mathbb{F}_u$.

\begin{theorem}\label{6}
 If $b^{(u)}_{\pi\theta}$ is an entry of $\mathbb{G}_u$, determined by elements $f[\theta],f[\pi]\in B'$, then $b^{(u)}_{\pi\theta}$ is noninvertible.
\end{theorem}

\begin{proof}
 Since $f[\theta],f[\pi]\in B'$, $[\theta],[\pi]\in A'$ and thus, the entry $a^{(u)}_{\pi\theta}$ of $\mathbb{F}_u$ is noninvertible. Now, by Theorem
 \ref{5}(ii), $b^{(u)}_{\pi\theta}$ is noninvertible.
\end{proof}

\begin{theorem}\label{7}
 Let $m'$ be a multidegree that occurs in $\mathbb{T}_{M_m}$.
 \begin{enumerate}[(i)]
  \item There are no basis elements of $\mathbb{G}_u$, with multidegree $mm'$ and homological degree less than $j$.
  \item For every $i=0,\ldots,c$, there is a bijective correspondence between the basis elements of $\mathbb{F}_u$, with multidegree $m'$ and homological
  degree $i$, and the basis elements of $\mathbb{G}_u$, with multidegree $mm'$ and homological degree $i+j$.
 \end{enumerate}
\end{theorem}

\begin{proof}
 (i) Since the basis of $\mathbb{G}_u$ is contained in that of $\mathbb{T}_M$, the statement follows from Theorem \ref{3}(i).\\
 (ii) The set of basis elements of $\mathbb{F}_u$, with multidegree $m'$ and homological degree $i$ is 
 $A'_{i,m'}=A_{i,m'}\setminus\{[\theta_0],[\pi_0],\ldots,[\theta_{u-1}],[\pi_{u-1}]\}$. Similarly, the set of basis elements of $\mathbb{G}_u$, with 
 multidegree $mm'$ and homological degree $i+j$ is $B'_{i,m'}=B_{i,m'}\setminus\{f[\theta_0],f[\pi_0],\ldots,f[\theta_{u-1}],f[\pi_{u-1}]\}$. Notice that 
 $[\theta_k] \in A_{i,m'}$ if and only if $f[\theta_k] \in B_{i,m'}$. Likewise, $[\pi_k] \in A_{i,m'}$ if and only if $f[\pi_k] \in B_{i,m'}$.
 Therefore, if we restrict 
 $f_{i,m'}:A_{i,m'}\rightarrow B_{i,m'}$ to $A'_{i,m'}$, we get a bijection between $A'_{i,m'}$ and $B'_{i,m'}$.
\end{proof}

Notation: If $b^{(u)}_{\gamma_0\delta_0}$ is an invertible entry of $\mathbb{G}_u$, determined by basis elements
$[\delta_0],[\gamma_0]$ of $\mathbb{G}_u$, let $\mathbb{G}_{u+1}$ be the resolution of $S/M$ such that
\[\mathbb{G}_u=\mathbb{G}_{u+1}\oplus\left(0\rightarrow S[\delta_0]\rightarrow S[\gamma_0]\rightarrow0\right).\]
Assume that $\mathbb{G}_{u+(k-1)}$ has been defined. If $b^{(u+k-1)}_{\gamma_{k-1}\delta_{k-1}}$ is an invertible entry of $\mathbb{G}_{u+(k-1)}$, determined by basis elements $[\delta_{k-1}]$, $[\gamma_{k-1}]$ of 
$\mathbb{G}_{u+(k-1)}$, let 
$\mathbb{G}_{u+k}$ be the resolution of $S/M$ such that 
\[\mathbb{G}_{u+(k-1)}=\mathbb{G}_{u+k}\oplus \left(0\rightarrow S[\delta_{k-1}]\rightarrow S[\gamma_{k-1}]\rightarrow 0\right).\]

\begin{theorem}\label{8}
 Suppose that $\mathbb{G}_u$, $\mathbb{G}_{u+1}$ are defined as above. If $b^{(u)}_{\pi\theta}$ is an entry of $\mathbb{G}_u$, determined by 
 elements $f[\theta],f[\pi]\in B'$, then $f[\theta],f[\pi]$ are in the basis of $\mathbb{G}_{u+1}$. Moreover, the entry $b^{(u+1)}_{\pi\theta}$ of 
 $\mathbb{G}_{u+1}$, determined by $f[\theta]$ and $f[\pi]$ is noninvertible.
\end{theorem}

\begin{proof}
 Since $\mathbb{G}_u=\mathbb{G}_{u+1}\oplus\left(0\rightarrow S[\delta_0]\rightarrow S[\gamma_0]\rightarrow0\right)$, we have that 
 $b^{(u)}_{\gamma_0\delta_0}$ is invertible, $[\delta_0],[\gamma_0]$ are in $\mathbb{G}_u$, in consecutive homological degrees, and 
 $\mdeg[\delta_0]=\mdeg[\gamma_0]$. Suppose that $[\delta_0],[\gamma_0]\in B'$. Then there are elements $[\sigma],[\tau]\in A'$, in consecutive homological 
 degrees, such that $f[\sigma]=[\delta_0]$ and $f[\tau]=[\gamma_0]$. By Theorem \ref{7}(ii), $[\sigma],[\tau]$ are in $\mathbb{F}_u$ and determine and entry 
 $a^{(u)}_{\tau\sigma}$. Now, it follows from Theorem \ref{5}(ii), that $b^{(u)}_{\gamma_0\delta_0}=(-1)^ja^{(u)}_{\tau\sigma}$. This means that 
 $a^{(u)}_{\tau\sigma}$ is invertible and $\mathbb{F}_u$ is not minimal, a contradiction. 
 On the other hand, if only one of $[\delta_0],[\gamma_0]$ is in $B'$, then $\mdeg[\delta_0]\neq \mdeg[\gamma_0]$; another contradiction.
 We conclude that neither $[\delta_0]$ nor $[\gamma_0]$ is in $B'$ 
 and therefore, the pairs $\left(f[\theta],f[\pi]\right)$; $\left([\delta_0],[\gamma_0]\right)$ are disjoint. This proves that $f[\theta],f[\pi]$ are in 
 the basis of $\mathbb{G}_{u+1}$.
 
 Let us finally prove that $b^{(u+1)}_{\pi\theta}$ is noninvertible. If $\hdeg f[\theta]\neq\hdeg[\delta_0]$, then $b^{(u+1)}_{\pi\theta}=b^{(u)}_{\pi\theta}$ by [Al, Lemma 3.2(iv)], 
 and by Theorem \ref{6}, $b^{(u+1)}_{\pi\theta}$ is noninvertible. If $\hdeg f[\theta]=\hdeg[\delta_0]$, by [Al, Lemma 3.2(iii)], we have
 $b^{(u+1)}_{\pi\theta}=b^{(u)}_{\pi\theta}-\dfrac{b^{(u)}_{\pi\delta_0}b^{(u)}_{\gamma_0\theta}}{b^{(u)}_{\gamma_0\delta_0}}$. 
 Since $b^{(u)}_{\pi\theta}=(-1)^ja^{(u)}_{\pi\theta}$, $b^{(u)}_{\pi\theta}$ is noninvertible. Since $\mdeg f[\pi]\neq\mdeg[\delta_0]$, the entry 
 $b^{(u)}_{\pi\delta_0}$ of $\mathbb{G}_u$, determined by $[\delta_0]$, $f[\pi]$, is noninvertible. Hence, the product $b^{(u)}_{\pi\delta_0}.b^{(u)}_{\gamma_0\theta}$
 must be noninvertible. This means that the quotient
 $\dfrac{b^{(u)}_{\pi\delta_0}b^{(u)}_{\gamma_0\theta}}{b^{(u)}_{\gamma_0\delta_0}}$ is noninvertible. Finally, $b^{(u+1)}_{\pi\theta}$ is noninvertible, for the 
 difference of two noninvertible monomials is noninvertible.
\end{proof}

\begin{theorem}\label{9}
 Suppose that $\mathbb{G}_u,\ldots,\mathbb{G}_{u+v}$ are defined as above. If $b^{(u)}_{\pi\theta}$ is an entry of $\mathbb{G}_u$, determined by 
 elements $f[\theta],f[\pi]\in B'$, then $f[\theta],f[\pi]$ are in the basis of $\mathbb{G}_{u+v}$, and the entry $b^{(u+v)}_{\pi\theta}$ of 
 $\mathbb{G}_{u+v}$, determined by $f[\theta],f[\pi]$ is noninvertible.
\end{theorem}

\begin{proof}
 The proof is by induction on $v$. If $v=1$, the statement is the content of Theorem \ref{8}.
 Let us assume that the statement holds for $v-1$.
 Since $\mathbb{G}_{u+(v-1)}=\mathbb{G}_{u+v}\oplus\left(0\rightarrow S[\delta_{v-1}]\rightarrow S[\gamma_{v-1}]\rightarrow 0\right)$, it follows that the 
 entry $b^{(u+v-1)}_{\gamma_{v-1}\delta_{v-1}}$ of $\mathbb{G}_{u+(v-1)}$, determined by $[\delta_{v-1}],[\gamma_{v-1}]$, is invertible. If we had that 
 $[\delta_{v-1}],[\gamma_{v-1}]\in B'$, then, by induction hypothesis, $b^{(u+v-1)}_{\gamma_{v-1}\delta_{v-1}}$ would be noninvertible, a contradiction.
 On the other hand, if exactly one of $[\delta_{v-1}],[\gamma_{v-1}]$ were in $B'$, their multidegrees would be different, another contradiction. Hence, neither 
 $[\delta_{v-1}]$ nor $[\gamma_{v-1}]$ is in $B'$. This means that the pairs $\left(f[\theta],f[\pi]\right)$, $\left([\delta_{v-1}],[\gamma_{v-1}]\right)$ are 
 disjoint. Thus, $f[\theta],f[\pi]$ are in the basis of $\mathbb{G}_{u+v}$.\\
 Let us now prove that $b^{(u+v)}_{\pi\theta}$ is noninvertible. If $\hdeg f[\theta]\neq\hdeg[\delta_{v-1}]$, then $b^{(u+v)}_{\pi\theta}=b^{(u+v-1)}_{\pi\theta}$ by [Al, Lemma 3.2 (iv)],
 and the result follows from induction hypothesis. Now, if $\hdeg f[\theta]=\hdeg[\delta_{v-1}]$,
 \[b^{(u+v)}_{\pi\theta}=b^{(u+v-1)}_{\pi\theta}-\dfrac{b^{(u+v-1)}_{\pi\delta_{v-1}}b^{(u+v-1)}_{\gamma_{v-1}\theta}}{b^{(u+v-1)}_{\gamma_{v-1}\delta_{v-1}}}\]
 by [Al, Lemma 3.2(iii)].
 Notice that $b^{(u+v-1)}_{\pi\theta}$ is noninvertible, by induction hypothesis. Since $\mdeg f[\pi]\neq\mdeg [\delta_{v-1}]$, it follows that the entry 
 $b^{(u+v-1)}_{\pi\delta_{v-1}}$ of $\mathbb{G}_{u+(v-1)}$, determined by $[\delta_{v-1}],f[\pi]$, is noninvertible. This implies that the product 
 $b^{(u+v-1)}_{\pi\delta_{v-1}}b^{(u+v-1)}_{\gamma_{v-1}\theta}$ is noninvertible. Moreover, since $b^{(u+v-1)}_{\gamma_{v-1}\delta_{v-1}}$ is invertible, 
 the quotient
 $\dfrac{b^{(u+v-1)}_{\pi\delta_{v-1}}b^{(u+v-1)}_{\gamma_{v-1}\theta}}{b^{(u+v-1)}_{\gamma_{v-1}\delta_{v-1}}}$
 is noninvertible. Finally, $b^{(u+v)}_{\pi\theta}$ is noninvertible, for the difference of two noninvertible monomials is noninvertible.
\end{proof}
Since the process of making standard cancellations must eventually terminate, there is an integer $v\geq0$, such that 
$\mathbb{G}_u,\ldots,\mathbb{G}_{u+v}$ are defined as above, and $\mathbb{G}_{u+v}$ is a minimal resolution of $S/M$.

For the rest of this section, $v$ is such an integer and $G_{u+v}$ is such a minimal resolution. Moreover, the resolutions 
$\mathbb{G}_u,\ldots,\mathbb{G}_{u+v}$ are fixed for the rest of this section.

\begin{theorem}\label{10}
 Let $m'$ be a multidegree that occurs in $\mathbb{T}_{M_m}$. For each $i$, there is a bijective correspondence between the basis elements of $\mathbb{G}_u$, 
 with multidegree $mm'$ and homological degree $i+j$, and the basis elements of $\mathbb{G}_{u+v}$, with multidegree $mm'$ and homological degree $i+j$.
\end{theorem}

\begin{proof}
 Since the basis of $\mathbb{G}_{u+v}$ is contained in that of $\mathbb{G}_u$, every basis element of $\mathbb{G}_{u+v}$, with multidegree $mm'$ and 
 homological degree $i+j$ is in $\mathbb{G}_u$. Conversely, every basis element of $\mathbb{G}_u$, with multidegree $mm'$ and homological degree $i+j$ is in 
 $\mathbb{G}_{u+v}$, by Theorem \ref{9}.
\end{proof}

\begin{theorem}\label{11}
 Let $\mathbb{F}$ be a minimal resolution of $S/{M_m}$, and let $\mathbb{G}$ be a minimal free resolution of $S/M$. Let $m'$ be a multidegree
 that occurs in $\mathbb{T}_{M_m}$. Then
 \begin{enumerate}[(i)]
  \item There are no basis elements of $\mathbb{G}$, with multidegree $mm'$ and homological degree less than $j$.
  \item For each $i$, there is a bijective correspondence between the basis elements of $\mathbb{F}$, with multidegree $m'$ and homological
  degree $i$, and the basis elements of $\mathbb{G}$, with multidegree $mm'$ and homological degree $i+j$.
 \end{enumerate}
\end{theorem}

\begin{proof}
 (i) Since the basis of $\mathbb{G}$ is contained in that of $\mathbb{T}_M$, this part follows from Theorem \ref{3}(i).\\
 (ii) This part follows immediately from Theorem \ref{7}(ii) and Theorem \ref{10}.
\end{proof}

\begin{example}\label{12}
Consider Example \ref{example 1}, again. Recall that $M=(m_1,m_2,h_1,h_2,h_3,h_4,h_5)=(a^3b^2,c^3d,ac^2,a^2c,b^2d,abc,bcd)$, and $m=\lcm(m_1)=a^3b^2$. Since $M_m=(c^2,c,d,c,cd)=(c,d)$, the minimal resolution of $S/M_m$ is of the form
\[\mathbb{F}: 0\rightarrow S[c,d] \rightarrow 
\begin{array}{c} 
S[c]\\ 
\oplus \\ 
S[d] 
\end{array} 
\rightarrow S[\varnothing] \rightarrow S/M_m \rightarrow 0. \]
Thus, $\betti_{0,1}\left(S/M_m\right)=\betti_{1,c}\left(S/M_m\right)=\betti_{1,d}\left(S/M_m\right)=\betti_{2,cd}\left(S/M_m\right)=1$.
By Theorem \ref{11}(ii) (with $m=a^3b^2$, and $j=1$),
$\betti_{1,a^3b^2}\left(S/M\right)=\betti_{2,a^3b^2c}\left(S/M\right)=\betti_{2,a^3b^2d}\left(S/M\right)=\betti_{3,a^3b^2cd}\left(S/M\right)=1$. 
 By Theorem \ref{11}(i),
$\betti_{0,a^3b^2}\left(S/M\right)=\betti_{0,a^3b^2c}\left(S/M\right)=\betti_{0,a^3b^2d}\left(S/M\right)=\betti_{0,a^3b^2cd}\left(S/M\right)=1$. (In the next section we will give the entire list of multigraded Betti numbers of $S/M$.)    
\end{example}

\section{Structural Decomposition Theorems}

The notation below retains its meaning until the end of this section. 

Let $M$ be an ideal with minimal generating set $G=\{m_1,\ldots, m_q, n_1, \ldots, n_p\}$, where $m_1,\ldots,m_q$ are dominant and $n_1,\ldots,n_p$ are nondominant. Let $1\leq d \leq q$, and let $H=\{m_{d+1},\ldots,m_q,n_1,\ldots,n_p\}$. Then $G$ can be expressed in the form $G=\{m_1,\ldots,m_d,h_1,\ldots,h_c\}$, where $H=\{h_1,\ldots,h_c\}$.

\begin{itemize}
\item If $c>0$, let $C=\{(j,m)\in \mathbb{Z}^+ \times S:\text{ there are integers }1\leq r_1<\cdots <r_j \leq d \text{, such that }m=\lcm(m_{r_1},\ldots,m_{r_j})\} \bigcup \{(0,1)\}$. For each $(j,m) \in C$, let $M_m=(h'_1,\ldots,h'_c)$,
where $h'_i=\dfrac{\lcm(m,h_i)}{m}$.
\item If $c=0$, let $C=\{(0,1)\}$ and let $M_1=M$.
\end{itemize}

\begin{theorem}\label{1SD} 
For each integer $k$ and each monomial $l$,
\[\betti_{k,l} (S/M)= \sum\limits_{(j,m)\in C} \betti_{k-j,l/m}(S/M_m).\]
\end{theorem}

\begin{proof}
If $c=0$, the theorem is trivial. Let us consider the case $c>0$. \\
If $\betti_{k,l}(S/M)=0$, then $\sum\limits_{(j,m)\in C}\betti_{k-j,l/m }(S/M_m)=0$, by Theorem \ref{11}(ii).
 Suppose now that $\betti_{k,l}(S/M)\neq 0$. Then there is an element $[\tau]$ in the basis of a minimal resolution of $S/M$, such that $\hdeg[\tau]=k$ and 
$\mdeg[\tau]=l$. Let $m_{r_1}, \ldots,m_{r_j}$ be the dominant monomials that are contained in $[\tau]$, and such that $\{m_{r_1}, \ldots , m_{r_j}\}$ is a subset of $\{m_1, \ldots , m_d\}$. Since all basis elements of  $\mathbb{T}_M$ with equal multidegree must contain the same dominant monomials [Al, Lemma 4.3], every basis element of $\mathbb{T}_M$ in homological degree $k$ and multidegree $l$ must be of the form $[m_{r_1},\ldots,m_{r_j},h_{s_1},\ldots,h_{s_{k-j}}]$. Let $m=\lcm(m_{r_1},\ldots,m_{r_j})$. Then $(j,m)\in C$, and $\betti_{k,l}(S/M)=\betti_{k-j,l/m}(S/M_m)$, by Theorem \ref{11}(ii).\\
 We will complete the proof by showing that $\betti_{k-j',l/m}(S/M_{m'})=0$, for all $(j',m')\in C \setminus \{(j,m)\}$. Let $(j',m')\in C$. Then there are dominant monomials $m_{u_1},\ldots,m_{u_j}$, such that 
 $m'=\lcm(m_{u_1},\ldots,m_{u_{j'}})$. Suppose that $\betti_{k-j,l/m}(S/M_{m'})\neq 0$. Then $\mathbb{T}_{M_{m'}}$ has a basis element $[h'_{t_1},\ldots,h'_{t_{k-j'}}]$ with multidegree $l/m$. By Proposition \ref{1},
 $\l=m \mdeg[h'_{t_1},\ldots,h'_{t_{k-j'}}]=\mdeg[m_{u_1},\ldots,m_{u_{j'}},h_{t_1},\ldots,h_{t_{k-j'}}]$. Since the basis elements of $\mathbb{T}_M$ in homological degree $k$ and multidegree $l$ are of the form 
 $[m_{r_1},\ldots,m_{r_j},h_{s_1},\ldots,h_{s_{k-j}}]$, we must have that $\{m_{u_1},\ldots,m_{u_{j'}}\}=\{m_{r_1},\ldots,m_{r_j}\}$. In particular, $j'=j$, and 
 $m'=\lcm(m_{u_1},\ldots,m_{u_{j'}})=\lcm(m_{r_1},\ldots,m_{r_j})=m$. Thus $(j',m')=(j,m)$. 
\end{proof}

\begin{definition}
Recall that $G=\{m_1,\ldots, m_q, n_1, \ldots, n_p\}=\{m_1,\ldots,m_d,h_1,\ldots,h_c\}$ is the minimal generating set of $M$. If $d=q$, the equation 
\[\betti_{k,l} (S/M)= \sum\limits_{(j,m)\in C} \betti_{k-j,l/m}(S/M_m),\]
given by Theorem \ref{1SD}, will be called the \textbf{first structural decomposition} of $M$.
\end{definition}

Note that when $d=q$, we have that $c=p$, and $\{h_1,\ldots,h_c\}=\{n_1,\ldots,n_p\}$. 

\begin{example}\label{example 1SD}
Consider Example \ref{example 1}, again. Recall that $M=(m_1,m_2,n_1,n_2,n_3,n_4,n_5)=(a^3b^2,c^3d,ac^2,a^2c,b^2d,abc,bcd)$, where $\{h_1,\ldots,h_5\}=\{n_1,\ldots,n_5\}$, and hence, $d=q$. By definition, 
\begin{dmath*}
C=\{(2,\lcm(a^3b^2,c^3d));(1,\lcm(a^3b^2));(1,\lcm(c^3d));(0,1)\}=\{(2,a^3b^2c^3d);(1,a^3b^2);(1,c^3d);(0,1)\}. 
\end{dmath*}
Now, each ordered pair $(j,m)$ in $C$, determines a monomial ideal $M_m$. Namely, $(2,a^3b^2c^3d)$ defines $M_{a^3b^2c^3d}=(h'_1,h'_2,h'_3,h'_4,h'_5)$, where $h'_1=\dfrac{\lcm(a^3b^2c^3d,ac^2)}{a^3b^2c^3d}=1$. Therefore, $M_{a^3b^2c^3d}=(1)=S$.\\
Likewise, $(1,a^3b^2)$ defines $M_{a^3b^2}=(c,d)$ (recall Example \ref{12}). \\
Also, $(1,c^3d)$ defines $M_{c^3d}=(h'_1,h'_2,h'_3,h'_4,h'_5)$, where $h'_1=\dfrac{\lcm(c^3d,ac^2)}{c^3d}=a$; $h'_2=\dfrac{\lcm(c^3d,a^2c)}{c^3d}=a^2$; $h'_3=\dfrac{\lcm(c^3d,b^2d)}{c^3d}=b^2$; $h'_4=\dfrac{\lcm(c^3d,abc)}{c^3d}=ab$; 
$h'_5=\dfrac{\lcm(c^3d,bcd)}{c^3d}=b$. Thus, $M_{c^3d}=(a,a^2,b^2,ab,b)=(a,b)$.\\
Finally, $(0,1)$ defines $M_1=M=(ac^2,a^2c,b^2d,abc,bcd)$. Therefore, the first structural decomposition decomposition of $M$ is
\[
\betti_{k,l}(S/M)=\betti_{k-1,l/a^3b^2}\left(\dfrac{S}{(c,d)}\right)+\betti_{k-1,l/c^3d}\left(\dfrac{S}{(a,b)}\right)+\betti_{k,l}\left(\dfrac{S}{(ac^2,a^2c,b^2d,abc,bcd)}\right).\\ 
\]
\end{example}

\begin{theorem}\label{2SD}
 There is a family $\mathscr{D}$ of dominant ideals and a family $\mathscr{N}$ of purely nondominant ideals, such that 
\[\betti_{k,l}(S/M)=\sum\limits_{D\in \mathscr{D}} \betti_{k-j_D,l/m_D}(S/D) + \sum\limits_{N\in \mathscr{N}} \betti_{k-j_N,l/m_N}(S/N),\]
 where $j_D$, $j_N$ are integers that depend on $D$ and $N$, respectively, and $m_D$, $m_N$ are monomials that depend on $D$ and $N$, respectively.
\end{theorem}

\begin{proof}
Let  \[\betti_{k,l}(S/M)=\sum\limits_{(j,m)} \betti_{k-j,l/m}(S/M_m)\]
be the first structural decoposition of $M$. If some $M_m=(h'_1,\ldots,h'_p)$ (recall that $c=p$) is neither dominant nor purely nondominant, then its minimal generating set is of the form $\{u_1,\ldots, u_{q_1},v_1,\ldots,v_{p_1}\}$, where $u_1,\ldots,u_{q_1}$ are dominant, $v_1,\ldots,v_{p_1}$ are nondominant, $q_1 \geq 1$, $p_1 \geq 1$, and $q_1 +p_1 \leq p$. In particular, $p_1\leq p-1$. Let 
 $\betti_{k,l}(S/M_m)=\sum\limits_{(j',m')} \betti_{k-j',l/m'}(S/M_{m,m'})$ be the first structural decomposition of $M_m$. Combining the last two identities, we obtain 
\[\betti_{k,l}(S/M)=\sum\limits_{(j,m)}\betti_{k-j,l/m}(S/M_m) = \sum\limits_{(j,m)} \sum\limits_{(j',m')} \betti_{k-j-j',l/mm'}(S/M_{m,m'}).\]

If some $M_{m,m'}$ is neither dominant nor purely nondominant, then $M_{m,m'}=(v'_1,\ldots,v'_{p_1})$ (where $v'_i=\dfrac{\lcm(m',v_i)}{m'}$), and the number $p_2$ of nondominant generators in its minimal generating set is less than $p_1$ (because  $M_{m,m'}$ is minimally generated by at most $p_1$ monomials). In particular, $p_2\leq p_1-1\leq p-2$. Suppose that, after applying Theorem \ref{1SD} $r$ times, we obtain a decomposition 
 $\betti_{k,l}(S/M)=\sum\limits_{(j_1,l_1)} \cdots \sum\limits_{(j_r,l_r)} \betti_{k-j_1-\ldots -j_r,l/l_1\ldots l_r}(S/M_{l_1,\ldots,l_r})$, such that if some $M_{l_1,\ldots,l_r}$ is neither dominant nor purely nondominant, then 
 $M_{l_1,\ldots,l_r}=(w'_1,\ldots,w'_{p_{r-1}})$, with $p_{r-1}\leq p-(r-1)$. \\
 If some $M_{l_1,\ldots,l_r}$ is neither dominant nor purely nondominant, then the number $p_r$ of nondominant generators in its minimal generating set is less than $p_{r-1}$. In particular, 
 $p_r\leq p_{r-1}-1 \leq p-r$. Therefore, after applying Theorem \ref{1SD} $p$ times, we obtain a decomposition
 \[\betti_{k,l}(S/M)=\sum\limits_{(j_1,l_1)}\cdots \sum\limits_{(j_p,l_p)} \betti_{k-j_1-\ldots-j_p,l/l_1\ldots l_p}(S/M_{l_1,\ldots,l_p}).\]
 If we assume that there is an ideal $M_{l_1,\ldots,l_p}$ which is neither dominant nor purely nondominant, then $M_{l_1,\ldots,l_p}=(z'_1,\ldots,z'_{p_{p-1}})$, with $p_{p-1}\leq p-(p-1)=1$.\\
 But this scenario is not possible, for the minimal generating set of such an ideal must contain at least one dominant generator and at least one nondominant generator. \\
 We conclude that each $M_{l_1,\ldots,l_p}$ is either dominant or purely nondominant. 
\end{proof}

\begin{definition}
The equation
\[\betti_{k,l}(S/M)=\sum\limits_{D\in \mathscr{D}} \betti_{k-j_D,l/m_D}(S/D)+ \sum\limits_{N\in \mathscr{N}} \betti _{k-j_N,l/m_N}(S/N)\]
constructed in the proof of Theorem \ref{2SD} will be called \textbf{second structural decomposition} of $M$. The sum $\sum\limits_{D\in \mathscr{D}} \betti_{k-j_D,l/m_D}(S/D)$ will be called \textbf{dominant part} of the second structural decomposition, and the sum $\sum\limits_{N\in \mathscr{N}} \betti _{k-j_N,l/m_N}(S/N)$ will be called \textbf{purely nondominant part} of the second structural decomposition.
\end{definition}

\textit{Note}:
Although Theorem \ref{2SD} states the existence of a decomposition of the form
\[\betti_{k,l}(S/M)=\sum\limits_{D\in \mathscr{D}} \betti_{k-j_D,l/m_D}(S/D) + \sum\limits_{N\in \mathscr{N}} \betti_{k-j_N,l/m_N}(S/N),\]
the  proof of Theorem \ref{2SD} is constructive. In fact, we show that 
\[\betti_{k,l}(S/M)=\sum\limits_{(j_1,l_1)}\cdots \sum\limits_{(j_p,l_p)} \betti_{k-j_1-\ldots-j_p,l/l_1\ldots l_p}(S/M_{l_1,\ldots,l_p}),\]
where the ideals $M_{l_1,\ldots,l_p}$ are either dominant or purely nondominant, and they determine the dominant and purely nondominant part of the second structural decomposition.\\

Recall that if $D$ is a dominant ideal, its minimal resolution is given by $\mathbb{T}_D$ [Al, Theorem 4.4]. Therefore, when the second structural decomposition of $M$ has no purely nondominant part, we can immediately compute the multigraded Betti numbers $\betti_{k,l}(S/M)$. Such is the case in the next example.

\begin{example}\label{example 2SD}
In Example \ref{example 1} we introduced the ideal $M=(a^3b^2,c^3d,ac^2,a^2c,b^2d,abc,bcd)$ and, in Example \ref{example 1SD} we gave its first structural decomposition. We would like to read off the Betti numbers of $S/M$ from the Betti numbers of the three ideals on the right side of that decomposition. The first two of these ideals, namely $M_{a^3b^2}=(c,d)$ and $M_{c^3d}=(a,b)$, are dominant. Hence, their minimal  resolutions are $\mathbb{T}_{M_{a^3b^2}}$ and $\mathbb{T}_{M_{c^3d}}$, respectively. However, the third ideal, $M_1=(ac^2,a^2c,b^2d,abc,bcd)$, is not dominant. In order to obtain the multigraded Betti numbers of $S/M_1$ we compute the first structural decomposition of $M_1$ (we leave the details to the reader):\\
\begin{align*}
\betti_{k,l}\left(\dfrac{S}{M_1}\right)&=
\betti_{k-2,l/a^2c^2}\left(\dfrac{S}{(b)}\right)+\betti_{k-1,l/ac^2}\left(\dfrac{S}{(b)}\right)+ \betti_{k-1,l/a^2c}\left(\dfrac{S}{(b)}\right)
+ \betti_{k-1,l/b^2d}\left(\dfrac{S}{(c)}\right) \\
& +\betti_{k,l}\left(\dfrac{S}{(abc,bcd)}\right).
\end{align*}

Now, if we combine this equation with the first structural decomposition of $M$, given in Example \ref{example 1SD}, we obtain
\begin{align*}
\betti_{k,l}(S/M) & =
\betti_{k-1,l/a^3b^2}\left(\dfrac{S}{(c,d)}\right)+\betti_{k-1,l/c^3d}\left(\dfrac{S}{(a,b)}\right)+\betti_{k-2,l/a^2c^2}\left(\dfrac{S}{(b)}\right)\\
&+
\betti_{k-1,l/ac^2}\left(\dfrac{S}{(b)}\right) 
 + \betti_{k-1,l/a^2c}\left(\dfrac{S}{(b)}\right) 
 + \betti_{k-1,l/b^2d}\left(\dfrac{S}{(c)}\right) + \betti_{k,l}\left(\dfrac{S}{(abc,bcd)}\right).
\end{align*}

Note that this is the second structural decomposition of $M$, for each ideal on the right side of this decomposition is dominant. In order to compute $\betti_{k,l}(S/M)$, it would be unwise to choose integers $k$ and monomials $l$ at random. We might take many guesses and still not find any nonzero multigraded Betti numbers. The right way to compute $\betti_{k,l}(S/M)$ is by first computing the minimal resolutions of the dominant ideals on the right side of the decomposition, which we do next.
\begin{itemize}
\item The multigraded Betti numbers of $S/(c,d)$ are \\
 $\betti_{0,1}(S/(c,d))=\betti_{1,c}(S/(c,d))=\betti_{1,d}(S/(c,d))=\betti_{2,cd}(S/(c,d))=1$. \\ 
 Therefore, $\betti_{k-1,l/a^3b^2}\left(\dfrac{S}{(c,d)}\right)=1$ when $(k-1,l/a^3b^2)$ equals one of $(0,1)$, $(1,c)$, $(1,d)$, $(2,cd)$; that is, when $(k,l)$ equals one of $(1,a^3b^2)$, $(2,a^3b^2c)$, $(2, a^3b^2d)$, $(3,a^3b^2cd)$.
\item The multigraded Betti numbers of $S/(a,b)$ are \\
 $\betti_{0,1}(S/(a,b))=\betti_{1,a}(S/(a,b))=\betti_{1,b}(S/(a,b))=\betti_{2,ab}(S/(a,b))=1$. \\ 
 Therefore, $\betti_{k-1,l/c^3d}\left(\dfrac{S}{(a,b)}\right)=1$ when $(k-1,l/c^3d)$ equals one of $(0,1)$, $(1,a)$, $(1,b)$, $(2,ab)$; that is, when $(k,l)$ equals one of $(1,c^3d)$, $(2,ac^3d)$, $(2, bc^3d)$, $(3,abc^3d)$.
\item  The multigraded Betti numbers of $S/(b)$ are
 $\betti_{0,1}(S/(b))=\betti_{1,b}(S/(b))=1$. \\ 
 Therefore, $\betti_{k-2,l/a^2c^2}\left(\dfrac{S}{(b)}\right)=1$, or $\betti_{k-1,l/ac^2}\left(\dfrac{S}{(b)}\right)=1$, or $\betti_{k-1,l/a^2c}\left(\dfrac{S}{(b)}\right)=1$, when $(k-2,l/a^2c^2)$ equals one of $(0,1)$, $(1,b)$, or when $(k-1,l/ac^2)$ equals one of $(0,1)$, $(1,b)$, or when $(k-1,l/a^2c)$ equals one of $(0,1)$, $(1,b)$; that is, when $(k,l)$ equals one of $(2,a^2c^2)$, $(3,a^2bc^2)$, $(1, ac^2)$, $(2,abc^2)$, $(1, a^2c)$, $(2,a^2bc)$.
\item  The multigraded Betti numbers of $S/(c)$ are
 $\betti_{0,1}(S/(c))=\betti_{1,c}(S/(c))=1$. \\ 
Therefore, $\betti_{k-1,l/b^2d}\left(\dfrac{S}{(c)}\right)=1$ when $(k-1,l/b^2d)$ equals one of $(0,1)$, $(1,c)$; that is, when $(k,l)$ equals one of $(1,b^2d)$, $(2,b^2cd)$.
\item The multigraded Betti numbers of $S/(abc,bcd)$ are \\
 $\betti_{0,1}(S/(abc,bcd))=\betti_{1,abc}(S/(abc,bcd))=\betti_{1,bcd}(S/(abc,bcd))=\betti_{2,abcd}(S/(abc,bcd))=1$. \\ 
 Therefore, $\betti_{k,l}\left(\dfrac{S}{(abc,bcd)}\right)=1$ when $(k,l)$ equals one of $(0,1)$, $(1,abc)$, $(1, bcd)$, $(2,abcd)$.
\end{itemize}
Thus, the nonzero multigraded Betti numbers of $S/M$ are\\
$\betti_{3,a^3b^2cd}=\betti_{3,abc^3d}=\betti_{3,a^2bc^2}=\betti_{2,a^3b^2c}=\betti_{2,a^3b^2d}=\betti_{2,ac^3d}=\betti_{2,bc^3d}=\betti_{2,a^2c^2}=\betti_{2,abc^2}=\betti_{2,a^2bc}=\betti_{2,b^2cd}=\betti_{2,abcd}=\betti_{1,a^3b^2}=\betti_{1,c^3d}=\betti_{1,ac^2}=\betti_{1,a^2c}=\betti_{1,b^2d}=\betti_{1,abc}=\betti_{1,bcd}=\betti_{0,1}=1.$
 \end{example}
 
\begin{definition}
Recall that $G=\{m_1,\ldots,m_d,h_1,\ldots,h_c\}$ is the minimal generating set of $M$. If $d=1$, the equation 
\[\betti_{k,l} (S/M)= \sum\limits_{(j,m)\in C} \betti_{k-j,l/m}(S/M_m),\]
given by Theorem \ref{1SD}, will be called the \textbf{third structural decomposition} of $M$.
\end{definition}

Note that when $d=1$, the right hand side of the equation above has only two terms. The third strutural decomposition will be instrumental in the proof of Charalambous theorem, in Section 6.

\section{Decompositions without purely nondominant part}

When the second structural decomposition of $M$ has no purely nondominant part, the numbers $\betti_{k,l} (S/M)$ can be easily computed, as illustrated in Example \ref{example 2SD}. In this section, however, our aim is to compute Betti numbers of classes of ideals rather than single ideals. More specifically, we will introduce two families of ideals whose decompositions have no purely nondominant part, and will give their multigraded Betti numbers explicitly.

\begin{definition} 
Let $L$ be the set of all monomials $l$ such that the number of basis elements of $\mathbb{T}_M$, with multidegree $l$ is odd.  
\begin{enumerate}[(i)]
\item We say that $M$ has \textbf{characteristic Betti numbers}, if for each monomial $l$ 
\begin{equation*}
\sum\limits_k \betti_{k,l}(S/M)= 
 \begin{cases}
							1 \text{ if } l \in L, \\ 
							0 \text{ otherwise.}  
								\end{cases}
\end{equation*}								
\item For each $l \in L$, let
\[f(l)=\min\{\hdeg[\sigma]:[\sigma]\in \mathbb{T}_M \text{ and }	\mdeg[\sigma]=l\}.\]	
We say that $M$ has \textbf{characteristic Betti numbers in minimal homological degrees}, if 	
\begin{equation*}
\betti_{k,l}(S/M)=  \begin{cases} 
							1 \text{ if } l \in L \text{ and } k=f(l),\\ 
							0 \text{ otherwise.}  
							\end{cases}
\end{equation*}
																
\end{enumerate}
\end{definition}

\begin{lemma} \label{lemma} 
Let \[\betti_{k,l}(S/M)=\sum\limits_{(j,m)} \betti_{k-j,l/m}(S/M_m)\] be the first structural decomposition of $M$. Then, the second structural decomposition of $M$ has no purely nondominant part if and only if the second structural decomposition of each $M_m$ has no purely nondominant part.
\end{lemma}

\begin{proof}
Let
\[\betti_{k,l}(S/M_m)=\sum\limits_{D\in \mathscr{D}} \betti_{k-j_D,l/m_D}(S/D) + \sum\limits_{N\in \mathscr{N}} \betti_{k-j_N,l/m_N}(S/N),\]
be the second structural decomposition of $M_m$. Then\\
$\betti_{k,l}(S/M)= \sum\limits_{(j,m)} \betti_{k-j,l/m}(S/M_m)=$
\[\sum\limits_{(j,m)}\left( \sum\limits_{D\in \mathscr{D}_m}\betti_{k-j-j_D,l/mm_D}(S/D)+ \sum\limits_{N\in\mathscr{N}_m} \betti_{k-j-j_N,l/mm_N}(S/N) \right)\]
is the second structural decomposition of $M$. 
\end{proof}

\begin{theorem}\label{char Betti numbers} 
If the second structural decomposition of $M$ has no purely nondominant part, then $M$ has characteristic Betti numbers.
\end{theorem}

\begin{proof}
The proof is by induction on the cardinality of the minimal generating set $G$ of $M$.\\
If $\#G=1$ or $\#G=2$, then $M$ is dominant and, by [Al, Corollary 4.5], $M$ is Scarf. Now, Scarf ideals have characteristic Betti numbers.\\
Suppose now that the theorem holds for ideals with minimal generating sets of cardinality $\leq q-1$. \\
Let us assume that $\#G=q$. By hypothesis, the second structural decomposition of $M$ has no purely nondominant part, which implies that $M$ itself is not purely nondominant. Therefore, $G$ must be of the form $G=\{m_1,\ldots,m_s,n_1,\ldots,n_t\}$, where $m_1,\ldots,m_s$ are dominant, $n_1,\ldots,n_t$ are nondominant, $s>0$, and $s+t=q$. In particular, $t\leq q-1$. Now, the first structural decomposition of $M$ is
\[\betti_{k,l} (S/M)= \sum\limits_{(j,m)\in C} \betti_{k-j,l/m}(S/M_m),\]
where each $M_m$ is minimally generated by at most $q-1$ monomials.
Then,
\[\sum \limits_k \betti_{k,l}(S/M)=\sum\limits_k \sum\limits_{(j,m)} \betti_{k-j,l/m} (S/M_m) = \sum\limits_{(j,m)} \sum\limits_k \betti_{k-j,l/m} (S/M_m).\]
Suppose that, for some monomial $l$, $\sum\limits_k\betti_{k,l}(S/M) \neq 0$. Then, there must be a pair $(j',m') \in C$
such that $\sum\limits_k  \betti_{k-j',l/m'} (S/M_{m'}) \neq 0$. By Lemma \ref{lemma}, the second structural decomposition of $M_{m'}$ has no purely nondominant part and, by induction hypothesis, 
$M_{m'}$ has characteristic Betti numbers. Hence, $\sum\limits_k  \betti_{k-j',l/m'} (S/M_{m'}) =1$.\\
Suppose, by means of contradiction, that there is a pair $(j,m)\in C \setminus \{(j',m')\}$ such that $\sum\limits_k  \betti_{k-j,l/m} (S/M_m) \neq 0$. Then, there exist basis elements $[\sigma] \in \mathbb{T}_{M_{m'}}$ and $[\tau] \in \mathbb{T}_{M_m}$, such that $\mdeg[\sigma]=l/m'$, and $\mdeg[\tau]=l/m$. Recall that $[\sigma]$, $[\tau]$, $m'$, and $m$ are of the form $[\sigma]=[n'_{a_1},\ldots,n'_{a_c}]$; $[\tau]=[n'_{b_1},\ldots,n'_{b_d}]$; $m'=\lcm(m_{u_1},\ldots,m_{u_{j'}})$; $m=\lcm(m_{v_1},\ldots,m_{v_j})$. By Proposition \ref{1} we have that 
\[l=m'l/m'=m'\mdeg[\sigma]=m'\mdeg[n'_{a_1},\ldots,n'_{a_c}]=\mdeg[m_{u_1},\ldots,m_{u_{j'}},n_{a_1},\ldots,n_{a_c}].\]
Similarly, 
\[l=ml/m=m\mdeg[\tau]=m\mdeg[n'_{b_1},\ldots,n'_{b_d}]= \mdeg[m_{v_1},\ldots,m_{v_j},n_{b_1},\ldots,n_{b_d}].\]
 Hence, 
$\mdeg[m_{u_1},\ldots,m_{u_{j'}},n_{a_1},\ldots,n_{a_c}] $ and $\mdeg[m_{v_1},\ldots,m_{v_j},n_{b_1},\ldots,n_{b_d}]$ are two basis elements of $\mathbb{T}_M$, with the same multidegree $l$. By [Al, Lemma 4.3], these basis elements must contain the same dominant generators. However, since $(j',m')\neq(j,m)$, we must have that $\{m_{u_1},\ldots,m_{u_j'}\} \neq \{m_{v_1},\ldots,m_{v_j}\}$, a contradiction. Therefore, 
$\sum\limits_k  \betti_{k-j,l/m} (S/M_m)=0$, for all $(j,m)\in C \setminus\{(j',m')\}$. Thus,
\[\sum\limits_k \betti_{k,l}(S/M)= \sum\limits_{(j,m)} \sum\limits_k \betti_{k-j,l/m} (S/M_m)= \sum\limits_k \betti_{k-j',l/m'} (S/M_{m'})=1.\]
We have proven that $\sum\limits_k \betti_{k,l}(S/M)\leq 1$, for each monomial $l$.\\
Since a minimal resolution $\mathbb{F} $ of $S/M$ can be obtained from $\mathbb{T}_M$ by making series of consecutive cancellations, and given that each consecutive cancellation involves a pair of basis elements of equal multidegree, the number of basis elements of $\mathbb{T}_M$ with a given multidegree $l$ is even if and only if the number of basis elements of $\mathbb{F}$ with multidegree $l$ is even. But the number of basis elements of $\mathbb{F}$ with multidegree $l$ is $\sum\limits_k \betti_{k,l}(S/M) \leq 1$, which proves the theorem.
\end{proof}

In Example \ref{example 2SD}, we computed the second structural decomposition of the ideal $M=(a^3b^2,c^3d,ac^2,a^2c,b^2d,abc,bcd)$, and noticed that it has no purely nondominant part. Right after, we found the numbers $\betti_{k,l}(S/M)$ and proved that, with the language of this section, $M$ has characteristic Betti numbers. This is consistent with Theorem \ref{char Betti numbers}. 

\begin{lemma}\label{hom degree} 
Let $\betti_{k,l}(S/M)=\sum\limits_{(j,m)} \betti_{k-j,l/m} (S/M_m)$ be the first structural decomposition of $M$. Suppose that $M$ has characteristic Betti numbers, and each $M_m$ has characteristic Betti numbers in minimal homological degrees. Then $M$ has characteristic Betti numbers in minimal homological degrees. 
\end{lemma}

\begin{proof}
Let $l$ be a monomial that is the common multidegree of an odd number of basis elements of $\mathbb{T}_M$. Let $r$ be such that $\betti_{r,l}(S/M)=1$. Then, there is a pair $(j,m)\in C$ such that $\betti_{r-j,l/m}(S/M_m)=1$. It follows that there is a basis element $[n'_{a_1},\ldots,n'_{a_{r-j}}]$ of $\mathbb{T}_{M_m}$ with multidegree $l/m$.
Recall that $m$ is of the form $m=\lcm(m_{b_1},\ldots,m_{b_j})$, with $m_{b_1},\ldots,m_{b_j} \in \{m_1,\ldots,m_s\}$. By Proposition \ref{1}, $l=m \dfrac{l}{m}=m \mdeg[n'_{a_1},\ldots,n'_{a_{r-j}}]=\mdeg[m_{b_1},\ldots,m_{b_j},n_{a_1},\ldots,n_{a_{r-j}}]$. Suppose that $[\sigma]$ is a basis element of $\mathbb{T}_M$, such that $\mdeg[\sigma]=l$. We will show that $r \leq \hdeg[\sigma]$. By [Al, Lemma 4.3], $[\sigma]$ must be of the form 
$[\sigma]=[m_{b_1},\ldots,m_{b_j},n_{c_1},\ldots,n_{c_d}]$. By Proposition \ref{1}, $\l=\mdeg[\sigma]=m \mdeg[n'_{c_1}\ldots,n'_{c_d}]$. Thus, the basis element $[n'_{c_1},\ldots,n'_{c_d}]$ of $\mathbb{T}_{M_m}$ has multidegree $l/m$. Since $M_m$ has Betti numbers in minimal homological degrees, $\hdeg[n'_{a_1},\ldots,n'_{a_{r-j}}]\leq \hdeg[n'_{c_1},\ldots,n'_{c_d}]$. It follows that $r=\hdeg[m_{b_1},\ldots,m_{b_j},n_{a_1},\ldots,n_{a_{r-j}}] \leq  \hdeg[m_{b_1},\ldots,m_{b_j},n_{c_1},\ldots,n_{c_d}] = \hdeg[\sigma]$, which proves that $M$ itself has characteristic Betti numbers in minimal homological degrees.
\end{proof}

\begin{definition} 
Suppose that the polynomial ring $S$ has $n$ variables $x_1,\ldots,x_n$. We will say that $M$ is \textbf{almost generic}, if there is an index $i$ such that no variable among 
$x_1,\ldots, x_{i-1}, x_{i+1}, \ldots,x_n$ appears with the same nonzero exponent in the factorization of two minimal generators of $M$.
\end{definition}

\begin{example} 
$M=(a^2b^2cd^2,a^3b^3c,cd^4)$ is almost generic because no variable among $a,b,d$ appears with the same nonzero exponent in the factorization of two minimal generators of $M$.
\end{example}

\begin{lemma}\label {Lemma 5.7} 
Let $\betti_{k,l}(S/M)=\sum\limits_{(j,m)} \betti_{k-j,l/m} (S/M_m)$ be the first structural decomposition of $M$. Suppose that $M$ is almost generic. Then, each $M_m$ is almost generic.
\end{lemma}

\begin{proof}
Let $G$ be the minimal generating set of $M$. By definition, there is an index $i$ such that no variable among $x_1,\ldots,x_{i-1},x_{i+1},\ldots,x_n$ appears with the same nonzero exponent in the factorization of two generators in $G$. 
Let $G=\{m_1,\ldots,m_s,n_1,\ldots,n_t\}$, where $m_1,\ldots,m_s$ are dominant, and $n_1,\ldots,n_t$ are nondominant. Then, each $M_m$ in the first structural decomposition of $M$ is of the form $M_m=(n'_1,\ldots,n'_t)$. Suppose, by means of contradiction, that some $M_m$ is not almost generic. Then, there is a variable $x \neq x_i$ that appears with the same nonzero exponent $\alpha$ in the factorization of two generators $n'_a,n'_b \in \{n'_1,\ldots,n'_t\}$. Recall that $n'_a=\dfrac{\lcm(m,n_a)}{m}$, $n'_b=\dfrac{\lcm(m,n_b)}{m}$. Let $u,v,w$ be the exponents with which x appears in the factorizations of $m,n_a,n_b$, respectively. Note that $x$ appears with exponents $u+\alpha$ and $\max(u,v)$ in the factorizations of $mn'_a$ and $\lcm(m,n_a)$, respectively. Since $mn'_a=\lcm(m,n_a)$, we must have that $u+\alpha=\max(u,v)$. It follows that $v=u+\alpha$. Likewise, $x$ appears with exponents $u+\alpha$ and $\max(u,w)$ in the factorizations of $mn'_b$ and $\lcm(m,n_b)$, respectively. Since $mn'_b=\lcm(m,n_b)$, we must have that $u+\alpha=\max(u,w)$. It follows that $w=u+\alpha$. Combining these identities, we deduce that $v=w$, which implies that $x$ appears with the same nonzero exponent in the factorizations of $n_a$ and $n_b$, which contradicts the fact that $M$ is almost generic.
\end{proof}

\begin{lemma} \label{almost generic lemma} 
If $M$ is almost generic, its second structural decomposition has no purely nondominant part.
\end{lemma}

\begin{proof}
Let $G$ be the minimal generating set of $M$. The proof is by induction on the cardinality of $G$.\\
If $\#G=1$ or $\#G=2$, then $M$ is dominant and the theorem holds.\\
Suppose that the theorem holds whenever $\#G\leq q-1$. Let us now assume that $\#G=q$. 
Since $M$ is almost generic, there is an index $i$ such that no variable among $x_1,\ldots,x_{i-1},x_{i+1},$\\
$\ldots,x_n$ appears with the same nonzero exponent in the factorization of two generators in $G$. Let $j \neq i$, and let $k$ be the greatest exponent with which $x_j$ appears in the factorization of a generator in $G$. Then there is a unique element in $G$, divisible by $x_j^k$, and such an element must be dominant (in $x_j$). Hence, $G$ can be represented in the form $G= \{m_1,\ldots,m_s,n_1,\ldots,n_t\}$, where $m_1,\ldots,m_s$ are dominant, $n_1,\ldots,n_t$ are nondominant, $s+t=q$, and $s \geq 1$. In particular $t \leq q-1$. By Lemma \ref{Lemma 5.7}, each $M_m$ in the first structural decomposition of $M$ is almost generic, and since $M_m=(n'_1,\ldots,n'_t)$, $M_m$ is minimally generated by at most $q-1$ elements. Then, by induction hypothesis, the second structural decomposition of $M_m$ has no purely nondominant part. Finally, the theorem follows from Lemma \ref {lemma}.
 \end{proof}

\begin{corollary}\label{Corollary 4} 
If $M$ is almost generic, it has characteristic Betti numbers.
\end{corollary}

\begin{proof}
Immediate from Theorem \ref{char Betti numbers} and Lemma \ref{almost generic lemma}.
\end{proof}

\begin{theorem}\label{MinHomDeg} 
If $M$ is almost generic, it has characteristic Betti numbers in minimal homological degrees.
\end{theorem}

\begin{proof}
By induction on the cardinality of the minimal generating set $G$ of $M$.\\
If $\#G=1$ or $\#G=2$, $M$ is dominant, and the Theorem holds. Let us assume that the theorem holds whenever $\#G \leq q-1$.\\
Suppose now that $\#G=q$. By Lemma \ref{almost generic lemma}, $M$ is not purely nondominant. Then, $G$ can be represented in the form $G=\{m_1,\ldots,m_s,n_1,\ldots,n_t\}$, where $m_1,\ldots,m_s$ are dominant, $n_1,\ldots,n_t$ are nondominant, $s+t=q$, and $s \geq 1$. In particular, $t \leq q-1$. Let $\betti_{k,l}(S/M)= \sum\limits_{(j,m)} \betti_{k-j,l/m} (S/M_m)$ be the first structural decomposition of $M$. Since $M_m=(n'_1,\ldots,n'_t)$, $M_m$ is minimally generated by at most $q-1$ monomials. By Lemma \ref{Lemma 5.7}, $M_m$ is almost generic. By induction hypothesis, $M_m$ has characteristic Betti numbers in minimal homological degrees. Now, the result follows from Lemma \ref{hom degree}.
\end{proof}

\begin{theorem} 
If $M$ is $2$-semidominant, it has characteristic Betti numbers in minimal homological degrees. 
\end{theorem}

\begin{proof}
Let $G=\{m_1,\ldots,m_q,n_1,n_2\}$ be the minimal generating set of $M$, where $n_1,n_2$ are nondominant. Then, the first structural decomposition of $M$ is
\[\betti_{k,l}(S/M)= \sum\limits_{(j,m)} \betti_{k-j,l/m} (S/M_m),\]
where $M_m=(n'_1,n'_2)$.
If $n'_1=1$ or $n'_2=1$, then $M_m=S$, and $\betti_{k-j,l/m} (S/M_m)=0$. Otherwise, $(n'_1,n'_2)$ is minimally generated by either one or two monomials, which implies that 
$M_m$ is dominant. Note that the first and second structural decompositions agree, and have no purely nondominant part. By Theorem \ref{char Betti numbers}, $M$ has characteristic Betti numbers. Moreover, since each $M_m$ is dominant, it has characteristic in minimal homological degrees. Now, the result follows from Lemma \ref{hom degree}.
\end{proof}

\section{Structural decomposition and Projective Dimension}

If the minimal generating set $G$ of $M$ has at least two monomials, and the ring $S$ has $n$ variables, there are two natural bounds for the projective dimension $\pd\left(S/M\right)$, namely, $2\leq \pd\left(S/M\right) \leq n$. In this section we discuss some cases where the lower and upper bounds are achieved.

Hilbert-Burch theorem [Ei, Theorem 20.15] describes the structure of the minimal resolutions of ideals $M$, when $\pd(S/M)=2$. The next theorem gives sufficient conditions for the lower bound $\pd(S/M)=2$ to be achieved.

\begin{theorem}\label{pd}
Let $M$ be either $2$- or $3$-semidominant. Suppose that there exist a minimal generator of $M$ that divides the $\lcm$ of every pair of minimal generators of $M$. Then $\pd(S/M)=2$.
\end{theorem}

\begin{proof}
Let $G=\{m_1,\ldots,m_s,n_1,\ldots,n_t\}$ be the minimal generating set of $M$, where $m_1,\ldots,$\\
$m_s$ are dominant, and $n_1,\ldots,n_t$ are nondominant. By hypothesis, there is an element $n$ in $G$ that divides the $\lcm$ of every pair of elements in $G$. However, since each $m_i$ is dominant, $m_i\nmid \lcm(n_1,n_2)$. It follows that $n$ must be one of the $n_1,\ldots,n_t$; say $n=n_1$. Let 
$\betti_{k,l}(S/M)= \sum\limits_{(j,m)\in C} \betti_{k-j,l/m}(S/M_m)$ be the first structural decomposition of $M$. Let $(j,m)\in C$. We will prove that $\betti_{k-j,l/m}(S/M_m)=0$ for all $k \geq 3$, and all monomials $l$.\\
First, let us assume that $j \geq 2$. Then there are monomials $m_{i_1},\ldots,m_{i_j} \in G$ such that $m=\lcm(m_{i_1},\ldots,m_{i_j})$, and $M_m=(n'_1,\ldots,n'_t)$, where $n'_i= \dfrac{\lcm(m,n_i)}{m}$. In particular, 
$n_1\mid \lcm(m_{i_1},m_{i_2})$, and hence, $n_1 \mid m$. This implies that $\lcm(m,n_1)=m$, and thus $n'_1=\dfrac{\lcm(m,n_1)}{m}=1$. Therefore, $M_m=S$, and $\betti_{k-j,l/m}(S/M_m)=0$ for all $k \geq 3$ and all monomials $l$.\\
Now, let us assume that $j=1$. Then $m=m_{i_1}$, and by hypothesis, $n_1 \mid \lcm(m_{i_1},n_k)$, for all $k=2,\ldots,t$. Then $\lcm(m_{i_1},n_1) \mid (m_{i_1},n_k)$, for all $k=2,\ldots,t$, and therefore, $n'_1\mid n'_k$, for all $k=2,\ldots,t$. This means that $M_m=(n'_1,\ldots,n'_t)=(n'_1)$, and thus, $\pd(S/M_m)\leq 1$. It follows that $\betti_{k-j,l/m}(S/M_m)=\betti_{k-1,l/m}(S/M_m)=0$, for all $k\geq 3$, and all monomials $l$.
Finally, suppose that $j=0$. Then $m=1$, and $M_m=(n'_1,\ldots,n'_t)=(n_1,\ldots,n_t)$. Since $M$ is either $2$-semidominant or $3$-semidominant, $t=2$ or $t=3$. If $t=2$, then $M_m=(n_1,n_2)$, and thus 
$\betti_{k-j,l/m}(S/M_m)=\betti_{k,l}(S/M_m)=0$, for all $k\geq 3$, and all monomials $l$. On the other hand, if $t=3$, since $n_1 \mid \lcm(n_2,n_3)$, $M_m=(n_1,n_2,n_3)$ is not dominant, and we must have 
$\pd(S/M_m)\leq 2$. It follows that $\betti_{k-j,l}(S/M_m)=\betti_{k,l}(S/M_m)=0$, for all $k \geq 3$, and all monomials $l$. Therefore, for all $k \geq 3$, and all monomials $l$ the first structural decomposition of $M$ gives 
$\betti_{k,l}(S/M)=\sum\limits_{(j,m)\in C}\betti_{k-j,l/m}(S/M_m)=0$, which means that $\pd(S/M)\leq 2$. However, since $\#G\geq 2$, we must have $\pd(S/M)=2$.
\end{proof}

\begin{example}
Let 
\begin{align*}
m_1 & =a^3c^2d^2e^2f^2g^2        &     m_2 & =a^2b^3d^2e^2f^2g^2        &      m_3 & =a^2b^2c^3e^2f^2g^2\\
 m_4 & =a^2b^2c^2d^3f^2g^2       &    m_5 & =a^2b^2c^2de^3g^2            &      m_6 & =b^2c^2d^2e^2g^3\\
  n_1 & =abcdefg                          &      n_2 & =a^2b^2c^2d^2e^2f                  &     n_3 & = a^2b^2c^2d^2ef^2.
  \end{align*}
  
   Let $M=(m_1,\ldots,m_6)$; $M_2=(m_1,\ldots,m_6,n_1,n_2)$; $M_3=(m_1,\ldots,m_6,n_1,n_2,n_3)$. It is clear that $M$ is dominant; $M_2$ is $2$-semidominant, and $M_3$ is $3$-semidominant. Note that $n_1$ divides the $\lcm$ of every pair of monomials in $\{m_1,\ldots,m_6,n_1,n_2,n_3\}$. By Theorem \ref{pd}, $\pd(S/M)=6$; $\pd(S/M_2)=\pd(S/M_3)=2$. (We see how adding a few monomials to the minimal generating set can change the projective dimension dramatically.)
\end{example}
The fact that Artinian monomial ideals have maximum projective dimension (in the sense of Hilbert Syzygy theorem) was proven by Charalambous [Ch] (see also [Pe, Corollary 21.6]), using the radical of an ideal as main tool. Here we give an alternative proof of this fact that relies entirely on the first structural decomposition.

\begin{theorem}\label{Artinian}
If $M$ is Artinian in $S=k[x_1,\ldots,x_n]$, then $\pd(S/M)=n$.
\end{theorem}

\begin{proof}
By induction on $n$. If $n=1$, the result is trivial. Suppose that $\pd(S/M)=n-1$, for Artinian ideals $M$ in $n-1$ variables.\\
Let $M$ be Artinian in $S=k[x_1,\ldots,x_n]$. Then, for each $i=1,\ldots,n$, the minimal generating set $G$ of $M$ contains a monomial $x_i^{\alpha_i}$, $\alpha_i\geq 1$. Notice that each $x_i^{\alpha_i}$ is dominant. 
Let $\betti_{k,l}(S/M)=\sum\limits_{(j,m)\in C} \betti_{k-j,l/m}(S/M_m)$ be the third structural decomposition of $M$, where $m_1=x_1^{\alpha_1}$, and $H=G\setminus\{ x_1^{\alpha_1} \}$. Since $x_1^{\alpha_1}$ is dominant, $(1,x_1^{\alpha_1}) \in C$. Thus, $\betti_{k-1,l/x_1^{\alpha_1}}(S/M_{x_1^{\alpha_1}})$ is one of the terms on the right side of the third structural decomposition.\\
Let $G=\{x_1^{\alpha_1},\ldots,x_n^{\alpha_n}, l_1,\ldots,l_q\}$. By construction,
\begin{align*}
 M x_1^{\alpha_1} &=\left(\dfrac{\lcm(x_1^{\alpha_1},x_2^{\alpha_2})}{x_1^{\alpha_1}},\ldots,\dfrac{\lcm(x_1^{\alpha_1},x_n^{\alpha_n})}{x_1^{\alpha_1}},\dfrac{\lcm(x_1^{\alpha_1},l_1)}{x_1^{\alpha_1}},\ldots,\dfrac{\lcm(x_1^{\alpha_1},l_q)}{x_1^{\alpha_1}}\right)\\
 & =(x_2^{\alpha_2},\ldots,x_n^{\alpha_n},l'_1,\ldots,l'_q). 
 \end{align*}
Since $x_1^{\alpha_1}$ is dominant, $x_1$ does not appear in the factorization of $l'_i$. Thus, $M_{x_1^{\alpha_1}}$ is an Artinian ideal in $n-1$ variables. By induction hypothesis, $\pd(S/M_{x_1^{\alpha_1}})=n-1$. Therefore, there is a monomial $l$, such that $\betti_{n-1,l}(S/M_{x_1^{\alpha_1}}) \neq 0$. Let $l'=l x_1^{\alpha_1}$.Then, $\betti_{n-1,l'/x_1^{\alpha_1}}(S/M_{x_1^{\alpha_1}}) \neq 0$. Finally,
\[\betti_{n,l'}(S/M)=\sum\limits_{(j,m)\in C\setminus\{(1,x_1^{\alpha_1})\}} \betti_{k-j,l'/m}(S/M_m) + \betti_{n-1,l'/x_1^{\alpha_1}}(S/M_{x_1^{\alpha_1}}) \neq 0,\] \\
which implies that $\pd(S/M)=n$.
\end{proof}

\begin{theorem} \label{Gasharov}
Let $M$ be Artinian in $S=k[x_1,\ldots,x_n]$. Let $\mathbb{F}_M$ be a minimal resolution of $S/M$, obtained from $\mathbb{T}_M$ by means of consecutive cancellations. Then there is a basis element $[\sigma]$ of $\mathbb{F}_M$, such that $\hdeg[\sigma]=n$, and $\mdeg[\sigma]$ is divisible by each variable $x_1,\ldots,x_n$.
\end{theorem}

\begin{proof}
By Theorem \ref{Artinian}, there is a basis element $[\sigma]$ in the basis of $\mathbb{F}_M$, such that $\hdeg[\sigma]=n$. \\
By means of contradiction, suppose that the set $\{x_{j_1},\ldots,x_{j_i}\}$ of all variables dividing $\mdeg[\sigma]$ is a proper subset of $\{x_1,\ldots,x_n\}$; that is $i \leq n-1$.\\
Let $m=\mdeg[\sigma]$; let $G$ be the minimal generating set of $M$, and let $M_m$ be the ideal generated by $\{l \in G: l \mid m\}$. By [GHP, Theorem 2.1], there is a subcomplex $\left(\mathbb{F}_M\right)_{\leq m}$ of $\mathbb{F}_M$, such that $[\sigma]$ is a basis element of $\left(\mathbb{F}_M\right)_{\leq m}$, and $\left(\mathbb{F}_M\right)_{\leq m}$ is a minimal resolution of $S/M_m$. Therefore, $\pd (S/M_m) \geq \hdeg[\sigma]=n$. However, since $M_m$ is a monomial ideal in $k[x_{j_1},\ldots,x_{j_i}]$, it follows from Hilbert Syzygy theorem that $\pd(S/M_m)\leq i \leq n-1$, a contradiction.\\
We conclude that $\mdeg[\sigma]$ is divisible by $x_1,\ldots,x_n$.
\end{proof}

Now we are ready to prove Charalambous theorem.

\begin{theorem} \label{combinatorics}
Let $M$ be Artinian in $S=k[x_1,\ldots,x_n]$. Then, for all $i=0,\ldots,n$, $\betti_i(S/M) \geq {n \choose i }$.
\end{theorem}

\begin{proof}
Let $0 \leq i \leq n$. Let $\mathbb{X}_i$ be the class of all subsets of $\{x_1,\ldots,x_n\}$, of cardinality $i$. Then $\#\mathbb{X}_i= {n\choose i}$. Let $\{x_{j_1},\ldots,x_{j_i}\} \in \mathbb{X}_i$. Let $G$ be the minimal generating set of $M$, and let $G_{j_1,\ldots,j_i}$ and let $M_m$ be the monomial ideal generated by $G_{j_1,\ldots,j_i}$. Since $M$ is Artinian in $k[x_1,\ldots,x_n]$, $G$ contains generators of the form $x_{j_1}^{\alpha_{j_1}},\ldots,x_{j_i}^{\alpha_{j_i}}$, and hence, $M_m$ is Artinian in $k[x_{j_1},\ldots,x_{j_i}]$. By Theorem \ref{Gasharov}, there is a basis element $[\sigma_{j_1,\ldots,j_i}]$. of a minimal resolution $\left(\mathbb{F}_M\right)_{\leq m}$ of $S/M_m$, such that $\hdeg[\sigma_{j_1,\ldots,j_i}]=i$, and $\mdeg[\sigma_{j_1,\ldots,j_i}]$ is divisible by $x_{j_1},\ldots,x_{j_i}$. Since $M_m$ is an ideal in $k[x_{j_1},\ldots,x_{j_i}]$, $\mdeg[\sigma_{j_1,\ldots,j_i}]$ is not divisible by any variable of $\{x_1,\ldots,x_n\} \setminus \{x_{j_1},\ldots,x_{j_i}\}$. By [GHP, Theorem 2.1], $\left(\mathbb{F}_M\right)_{\leq m}$ can be regarded as a subcomplex of a minimal resolution $\mathbb{F}_M$ of $S/M$. Therefore, $[\sigma_{j_1,\ldots,j_i}]$ is a basis element of $\mathbb{F}_M$. Notice that if $\{x_{j_1},\ldots,x_{j_i}\}$ and $\{x_{k_1},\ldots,x_{k_i}\}$ are different elements of $\mathbb{X}_i$, the basis elements $[\sigma_{j_1,\ldots,j_i}]$ and $[\sigma_{k_1,\ldots,k_i}]$, determined by these sets must be different too. In fact, the sets of variables dividing $\mdeg[\sigma_{j_1,\ldots,j_i}]$ and $\mdeg[\sigma_{k_1,\ldots,k_i}]$ are $\{x_{j_1},\ldots,x_{j_i}\}$ and $\{x_{k_1},\ldots,x_{_i}\}$, respectively. It follows that $\betti_i(S/M) \geq \# \mathbb{X}_i = {n \choose i }$.
\end{proof}

\begin{theorem}\label{pd=n}
Let $M$ be an ideal in $S=k[x_1,\ldots,x_n]$, with minimal generating set $G$. Suppose that $G$ contains a subset of the form
\[G'=\{x_1^{\alpha_i}x_i^{\beta_i}: \alpha_i \geq 0, \text{ }\beta_i \geq1, \text{ for all }i=1,\ldots,n\}.\]
Then $\pd(S/M)=n$.
\end{theorem}

\begin{proof}
Let $G=G' \cup \{l_1,\ldots,l_q\}$. Let $\betti_{k,l}(S/M)=\sum\limits_{(j,m)\in C} \betti_{k-j,l/m} (S/M_m)$ be the first structural decomposition of $M$. Let $\gamma=\alpha_1+\beta_1$. Then $x_1^\gamma=x_1^{\alpha_1}x_1^{\beta_1}$ is a minimal generator in $G$. Notice that the exponent $\gamma$ with which $x_1$ appears in the factorization of $x_1^\gamma$, must be larger than the exponent with which $x_1$ appears in the factorization of any other minimal generator $l$ in $G$; otherwise, $l$ would be a multiple of $x_1^\gamma$. Hence, $x_1^\gamma$ is dominant in $G$; which implies that $(1,x_1^{\alpha_{11}}) \in C$. Thus, $\betti_{k-1,l/x_1^{\gamma}}(S/M_{x_1^{\gamma}})$ is one of the terms on the right side of the first structural decomposition. Now,
\begin{align*}
M_{x_1^{\gamma}} &=\left(\dfrac{\lcm(x_1^{\gamma},x_1^{\alpha_2}x_2^{\beta_2})}{x_1^{\gamma}}, \ldots, \dfrac{\lcm(x_1^{\gamma},x_1^{\alpha_n}x_n^{\beta_n})}{x_1^{\gamma}}, \dfrac{\lcm(x_1^{\gamma},l_1)}{x_1^{\gamma}}, \ldots, \dfrac{\lcm(x_1^{\gamma},l_q)}{x_1^{\gamma}}\right)\\
& =(x_2^{\beta_2},\ldots, x_n^{\beta_n},l'_1,\ldots,l'_q),
\end{align*}
where $l'_i=\dfrac{\lcm(x_1^{\gamma},l_i)}{x_1^{\gamma}}$.\\
Since $x_1^{\gamma}$ is dominant, $x_1$ does not appear in the factorization of $l'_i$. Thus, $M_{x_1^{\gamma}}$ is an Artinian monomial ideal in $k[x_2,\ldots,x_{n-1}]$. It follows that $\pd(S/M_{x_1^{\gamma}})=n-1$. Therefore, there is a monomial $l$, such that $\betti_{n-1,l}(S/M_{x_1^{\gamma}}) \neq 0$. Let $l'=l x_1^{\gamma}$. Then, $\betti_{n-1,l'/x_1^{\gamma}}(S/M_{x_1^{\gamma}}) \neq 0$. Finally,
\[\betti_{n,l'}(S/M)= \sum\limits_{(j,m)\in C \setminus \{(1,x_1^{\gamma})\}}  \betti_{n-j,l'/m} (S/M_m) + \betti_{n-1,l'/x_1^{\gamma}}(S/M_{x_1^{\gamma}}) \neq 0,\]
which implies that $\pd(S/M)=n$.
\end{proof}

\begin{example}
Let $M=(x_1^3,x_1x_2,x_1x_3,x_1x_4,x_1x_5,x_2x_4,x_3x_5)$. Then, the subset $G'=\{x_1^3,x_1x_2,x_1x_3,x_1x_4,x_1x_5\}$ of the minimal generating set of $M$, satisfies the hypotheses of Theorem \ref{pd=n}. Hence, $\pd(S/M)=5$.
\end{example}

The hypothesis of Theorem \ref{pd=n} is more common than it may seem. For instance, Artinian ideals satisfy this condition. The hypothesis of Theorem \ref{pd=n} is also satisfied if $M$ is the smallest Borel ideal containing a given monomial. A particular case of Theorem \ref{pd=n} is proven in [Ra].

\section{Conclusion}
We close this article with some remarks, questions, and conjectures. 

The structural decomposition is one of the very few techniques that allow us to compute Betti numbers by hand, not for arbitrary monomial ideals, but for a wide class of them. As a matter of fact, the ideal $M$ in Example \ref{example 2SD}, is minimally generated by $7$ monomials and even so we were able to compute the numbers $\betti_{k,l}(S/M)$. (Starting with $\mathbb{T}_M$, we could also calculate the $\betti_{k,l}(S/M)$ by means of consecutive cancellations. But since the basis of $\mathbb{T}_M$ contains $\sum\limits {7 \choose i }=128$ elements, and the basis of the minimal resolution of $S/M$ contains $20$ elements, we should make $\dfrac{128-20}{2}=54$ consecutive cancellations, which obviously requires the use of software.) 
  
On a different note, the structural decomposition of an ideal $M$ generates a finite family $\{M_m\}$ of ideals that usually has these two properties: a) the minimal generating set of each $M_m$ has smaller cardinality than the minimal generating set of $M$; b) if $M$ is an ideal in $S=k[x_1,\ldots,x_n]$, then $M_m$ is an ideal in a polynomial ring with less than $n$ variables. As a consequence, the structural decomposition works well when one wants to prove facts by induction. Theorems \ref {char Betti numbers} and \ref{MinHomDeg}, for instance, are proven by induction on the cardinaliy of the minimal generating set. On the other hand, Theorem \ref{Artinian} is proven by induction on the number of variables.

One last comment. Theorem \ref{char Betti numbers} does not depict the entire family of ideals with characteristic Betti numbers. In fact, $M=(a^2bc,b^2c^2,a^2b^2,abc^2)$ is purely nondominant and has characteristic Betti numbers. What other ideals have characteristic Betti numbers? The next two conjectures suggest some possibilities. 

\begin{conjecture}
Suppose that for some ideal $M$, there are indices $i< j$ such that no variable among $x_1,\ldots,x_{i-1},x_{i+1},\ldots,x_{j-1},,x_{j+1},\ldots,x_n$ appears with the same nonzero exponent in the factorization of two minimal generators. Then $M$ has characteristic Betti numbers.
\end{conjecture}

\begin{conjecture}
Let $M$ be minimally generated by $\{m_1,\ldots,m_q,n_1,\ldots,n_p\}$, where the $m_i$ are dominant and the $n_i$ are nondominant. If the ideal $M'=(n_1,\ldots,n_p)$ is almost generic, then $M$ has characteristic Betti numbers in minimal homological degrees.
\end{conjecture}

Finally, we conjecture that Theorem \ref{pd} admits a simple generalization.

\begin{conjecture}
If $M$ is minimally generated by more than one monomial, and there is a minimal generator that divides the $\lcm$ of every pair of generators, then $\pd(S/M)=2$.

\end{conjecture}

\bigskip

\noindent \textbf{Acknowledgements}: I am deeply indebted to my wife Danisa for her support, from beginning to end of this project. Selflessly, she spent many hours typing many versions of this paper. With her common sense and her uncommon wisdom, she helped me to select the contents and organize the material. Time and time again, she turned my frustration into motivation. She is a true gift from God.

\end{document}